\documentclass[a4paper,11pt]{amsart}
\usepackage{amsfonts,amssymb,epsfig,latexsym}
\usepackage{amsmath}
\usepackage[latin1]{inputenc}
\usepackage{color,layout}
\usepackage{ifthen}

\usepackage{pdfsync}
\usepackage{graphicx}
\usepackage{color}
\usepackage{pdfsync}
\date{}

\newcommand{\be}{\begin{equation}}
\newcommand{\ee}{\end{equation}}

\def\la{\langle}
\def\ra{\rangle}

\def\R{\mathbb{R}}
\def\C{\mathbb{C}}

\renewcommand{\Re}{\text{{\rm Re}\;}}
\renewcommand{\Im}{\text{{\rm Im}\;}}

\def\O{\mathcal{O}}

\newtheorem{theorem}{Theorem}[section]
\newtheorem{lemma}[theorem]{Lemma}
\newtheorem{proposition}[theorem]{Proposition}

\theoremstyle{definition}

\newtheorem{remark}[theorem]{Remark}

\numberwithin{equation}{section}

\title[multidimensional molecular predissociation]{Widths of highly excited resonances in multidimensional molecular predissociation}
\author[ 
Andr\'e MARTINEZ \& Vania SORDONI]{ 
Andr\'e MARTINEZ$ {}^1$ \& Vania SORDONI${}^1$
 }

\begin{document}

\maketitle 
\addtocounter{footnote}{1}
\footnotetext{{\tt\small Universit\`a di Bologna,  
Dipartimento di Matematica, Piazza di Porta San Donato, 40127
Bologna, Italy, 
andre.martinez@unibo.it }\\Partly supported by Universit\`a di Bologna, Funds for Selected Research Topics}  
\begin{abstract}
We investigate the simple resonances of a 2 by 2 matrix of n-dimensional semiclassical Shr\"odinger operators that interact through a first order differential operator. We assume that one of the two (analytic) potentials admits a well with non empty interior, while the other one is non trapping and creates a barrier between the well and infinity. Under a condition on the resonant state inside the well, we find an optimal lower bound on the width of the resonance. The method of proof relies on Carleman estimates, microlocal propagation of the microsupport, and a refined study of a non involutive double characteristic problem in the framework of Sj\"ostrand's analytic microlocal theory.
\end{abstract}  
\vfill


\baselineskip = 18pt 
\vfill\eject
\section{Introduction}
The mathematical study of molecular predissociation goes back to \cite{Kl}, where a general framework is set in order to study this physical phenomenon. This framework mainly consists of semiclassical matrix Schr\"odinger operators, where the semiclassical parameter $h$ is the square root of the ratio of electronic to nuclear masses, and where potential wells interact with unbounded classically allowed regions, giving rise to molecular resonances.

In the case of two interacting electronic levels, the standard Agmon-estimates method also provides a general upper bound for the widths of these resonances. However, only very few results exist concerning a lower bound of such quantities (let us recall that the physical importance of the resonance widths lies in the fact that their inverses represent the life-time of the corresponding metastable molecule).

To our knowledge, there are two cases only where such a lower bound is obtained. The first is when the resonance is close to the ground state of the bonding potential: see \cite{GrMa} where a complete semiclassical asymptotic of the width is obtained. The second one is when the resonance is highly excited and the dimension of the nuclei-space is one: see \cite{As}. In both cases, the lower bound is optimal, in the sense that it is of the same order of magnitude as the upper bound given by Agmon estimates.

The purpose of this paper is to extend Ashida's result \cite{As} to the case of multidimensional nuclei-space.

The geometrical structure of this situation is very similar to that of shape resonances, in the sense that one potential forms an energy well surrounded by a potential barrier (in which both potentials are involved), and immersed inside a sea corresponding to the second potential. And actually, both the strategy of the proof and the result (namely, the link between the width if the resonance and the size of the resonant state inside the well) are similar to that of \cite{DaMa} where excited shape resonances are studied. 

Roughly speaking, the starting point is the explicit link between the width if the resonance and the size of the resonant state inside the sea, and the strategy consists in relating reasoning by contradiction and in propagating the smallness of the resonant state from the inside of the sea up to the boundary of the well. In the case of shape resonances, this meant first crossing the shore (that is, the boundary of the sea), then propagating inside the barrier, and finally arriving at the boundary of the well. In the case of molecular predissociation, an extra difficulty appears: the crossing of the so called ``crest'', that is the place where the two potentials coincide. In contrast with the other types of propagation (that involve the principal symbol only), the crossing of the crest is very sensitive to the lower order interaction terms. Indeed, such a propagation would not take place without these terms, and the mathematical problem in itself has many to do with old results on propagation for scalar operators with non-involutive double-characteristics (see, e.g., \cite{LL, PeSo}). The difference is that here, after conjugation by the convenient weight, the principal symbols of two operators that are involved become complex-valued. Then, in order to prove the propagation, it becomes necessary to use a FBI transform (in the sense of \cite{Sj}), associated with a complex canonical transform, such that {\it both} operators take a simpler form.

Concerning the two other types of propagation (across the boundaries of the sea and of the well), their proofs mainly follow those of \cite{DaMa, Ma1}, with additional technical difficulties due to the fact that we deal with a matrix-operator (while those considered in \cite{DaMa, Ma1} are scalar).

The paper is organized as fallows: In Section 2, we describe our analytical and geometrical assumptions, and we recall some basic facts about the Agmon distance. Section 3 is devoted to the statement of our main results. In Section 4 we give some properties of the resonant state, and we establish the link between the resonance width and the size of the resonant state inside the sea. In Section 5 we specify this link through a contradiction argument, by transferring the smallness of the resonance width to the smallness of the resonant state inside the sea. Then this smallness is first propagated across the shore in Section 6, then across the crest in Section 7, and finally up to the well in Section 8. A converse result is proved in Section 9, and Section 10 concerns the possible examples of application.

\section{Assumptions}

We consider the semiclassical $2\times 2$ matrix Schr\"odinger operator,
\be
\label{operator}
P= 
\left(\begin{array}{cc}
P_1 & hW\\
hW^* & P_2
\end{array}\right) 
\ee
with,
$$
P_j =P_j(x, hD_x):= -h^2\Delta +V_j(x) \quad (j=1,2),
$$
where $x=(x_1,\dots ,x_n)$ is the current variable in $\R^n$ ($n\geq 1$),
$h>0$ denotes the semiclassical parameter,  $W=W(x,hD_x)$ is a first-order
semiclassical differential operator, and $W^*$ stands for the formal adjoint of $W$.

This kind of operator appears in the Born-Oppenheimer approximation of molecules, after reduction to an effective Hamiltonian (see \cite{KMSW, MaSo}). In that case, the quantity $h^2$ stands for the inverse of the mass of the nuclei.

{\bf Assumption 1.} {\sl The potentials $V_1$ and $V_2$ are smooth and 
bounded on $\R^n$,   and
satisfy,
\be
\label{assV1}
\begin{aligned}
&  \mbox{The set } U:= \{ V_1 \leq 0\} \mbox{ is compact and connected},\\
& \nabla V_1 \not =0 \mbox{ on the boundary } \partial U \mbox{ of } U, \mbox{ and } \liminf_{\vert x\vert\rightarrow\infty} V_1 >0;
\end{aligned}
\ee
\be
\label{assV2}
\begin{aligned}
& V_2\left|_{U}\right. >0\mbox{ and } E=0 \mbox{ is a non-trapping energy for } V_2; \\
&V_2 \; \mbox{has a strictly negative limit as} \; \vert x\vert \rightarrow \infty.
\end{aligned}
\ee
}
\begin{center}  
{\includegraphics[scale=0.40]{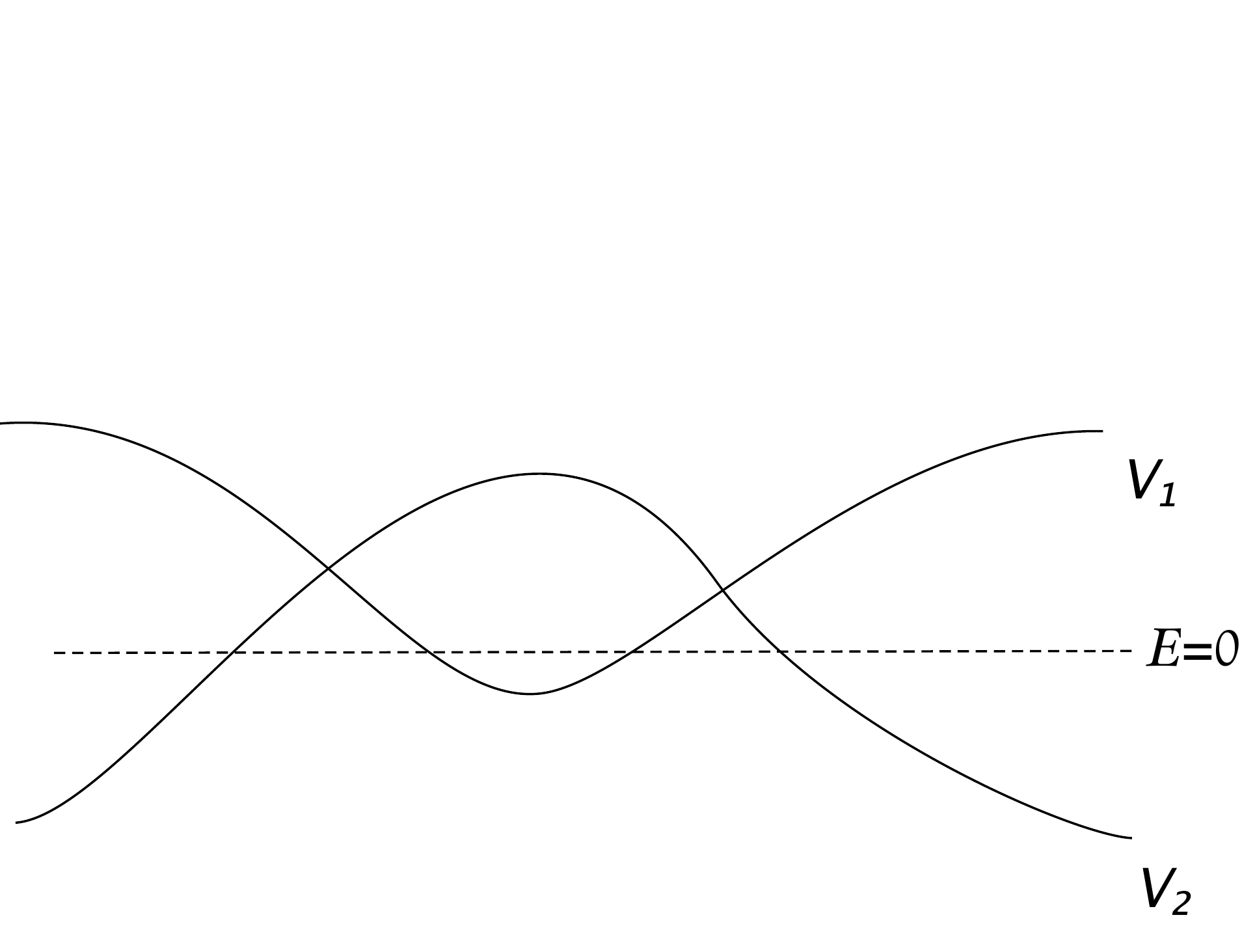}}\end{center}

We define the island  $\widehat{\mathcal I}$ as the bounded open connected component of  $\{V_2>0\}$ containing $U$,
and the sea as the unbounded closed set,
\be
\label{sea}
{\mathcal M}: = \R^n \backslash \widehat{\mathcal I},
\ee

With (\ref{assV1})-(\ref{assV2}), the well $U$ for $V_1$ is included in the island.

The fact that 0 is a non-trapping energy for $V_2$ means that, for any $(x,\xi)\in p_2^{-1}(0)$, one has $|\exp tH_{p_2}(x,\xi)|\rightarrow +\infty$ as $t\rightarrow \infty$, where $p_2(x,\xi):=\xi^2+V_2(x)$ is the symbol of $P_2$, and $H_{p_2}:=(\nabla_\xi \,p_2, -\nabla_x \,p_2)$ is the Hamilton field of $p_2$. 

Such conditions (\ref{assV1})-(\ref{assV2}) correspond to the situation of {\it molecular predissociation} as described in  \cite{Kl}.

We plan to study the resonances of $P$ near the energy level $E=0$. Since our methods will strongly rely on analytic microlocal analysis, we also add the following assumption of analyticity:

{\bf Assumption 2.}\label{Ass2} {\sl The potentials $V_1$ and $V_2$, together with the coefficients of $W$, extend to bounded holomorphic functions near a complex sector of the form,
${\mathcal S}_{\delta} :=\{x\in \C^n\, ;\,  \vert \Im x\vert \leq \delta |\Re x| \}$, with $\delta >0$. Moreover  $V_2$ tends to its limit at $\infty$ in this sector and $\Re V_1 $ stays away from $0$ outside a compact set of this sector.}

\begin{remark}\sl
As a matter of fact, it would have been possible for us to consider more general kinds of interactions, such as first-order pseudodifferential operators with analytic symbols (see, e.g., \cite{GrMa}). However, this would have just led to heavier notations and technical complications, without changing the key-ideas of the proof.
\end{remark}

Now, as in \cite{GrMa}, we define the cirque $\Omega_0$ as,
\be
\label{cirque}
\Omega_0 = \lbrace x \in \R^n; V_1 (x) < V_2 (x) \rbrace.
\ee
(Hence, the well is in the cirque and the cirque is in the island.)

Thereafter, the boundary $\partial \Omega_0$ of $\Omega_0$ will be called the crest, and the boundary $\partial{\mathcal M}$ of $\mathcal M$ will be called the shore.

 We also consider the Lithner-Agmon distance $d$ associated  to the pseudo-metric
 $$(\min(V_1, V_2))_+ dx^2.$$
 Such a metric is considered in \cite{Kl, GrMa, Pe}.  There are three places where this metric is not a standard smooth one.
 
At first, inside the well $U$ and the sea $\R^n\backslash \widehat{\mathcal I}$, where the metric degenerates completely.

Secondly on the crest $\partial \Omega_0$ (that is, at the points where $V_1 = V_2$). This case has been  considered in Pettersson \cite{Pe}.  At such points, if one assumes that $\nabla V_1 \not = \nabla V_2$, then  any geodesic that crosses transversally the hypersurface $\lbrace V_1= V_2 \rbrace$ is $C^1$ (and so is $\varphi$ near the crossing point).

Finally there are the boundaries of the island $\partial \widehat{\mathcal I}$ and of the well $\partial U$, where $\min (V_1,V_2)$ vanishes. This situation was considered in  \cite{HeSj2, Ma1}, and we will follow them in further constructions.

We denote by $\varphi$ the Lithner-Agmon distance to the well $U$,
\be
\label{agmondist}
\varphi (x) := d(U,x),
\ee
and we set,
\be
\label{S}
S_0 := d(U,{\mathcal M}).
\ee
We also consider the set,
$$
G:= \left\{ \overline{\gamma \cap (\widehat I \backslash U)}\, ;\, \gamma = \mbox{  minimal geodesics between } U \mbox{ and } {\mathcal M}\right\},
$$
which coincides with the set of minimal geodesics between $U$ and $\mathcal M$ that have only their end-points in $U\cup {\mathcal M}$. In particular, if $\gamma\in G$, then $\gamma \cap \mathcal M$ is a point of type 1  in the terminology of  \cite{HeSj2}.

Let us also recall that the assumption that $0$ is a non trapping energy for $V_2$ implies that $\nabla V_2  \not = 0$ on $\partial \widehat{\mathcal I}$, and therefore that $\partial \widehat{\mathcal I}$ is a smooth hypersurface.

\section{Main Results}
\label{secres}

Under the previous assumptions we plan to study the resonances of the operator $P$ given in (\ref{operator}).

In order to define the resonances we consider the distortion given as follows: Let $F(x) \in C^\infty (\R^n, \R^n)$ such that $F(x) = 0$ for $\vert x \vert \leq R_0$, $F(x) = x$ for $\vert x\vert$ large enough. For $\theta >0$ small enough, we define the distorded operator $P_{\theta}$ as the value at $ \nu = i \theta$ of  the extension to the complex of the  operator  $U_\nu P U_\nu^{-1}$,  defined for $\nu$ real small enough, and analytic in $\nu$, where we have set
\be
U_\nu \phi(x) = \det ( 1 + \nu dF(x))^{1/2} \phi ( x + \nu F(x)).
\ee 
By using the Weyl Perturbation Theorem, one can see  that there exists $\varepsilon_0 >0$ such that for any $\theta >0$ small enough, the spectrum of $P_\theta$ is discrete in $[ -\varepsilon_0, \varepsilon_0 ] - i [0, \varepsilon_0 \theta]$. The eigenvalues of $P_\theta$ are called the resonances of $P$ \cite{Hu, HeSj2, HeMa, Kl}.

In this paper we are interested in the imaginary part of these resonances. 

For some fixed $\delta >0$ arbitrarily small, let $\widetilde V_2$ be a $C^\infty$ function that coincides with $V_2$ on $\{ {\rm dist}(x,{\mathcal M})\geq \delta\}$, and such that $\inf_{\R^n}\widetilde V_2 >0$. By adapting the techniques used in \cite[Section 9]{HeSj2}, it is not difficult to see that the resonances of $P$ near 0 are exponentially close to the eigenvalues of the operator,
\be
\label{Ptilde}
\widetilde P := 
\left(\begin{array}{cc}
-h^2\Delta +V_1  & hW\\
hW^* & -h^2\Delta +\widetilde V_2
\end{array}\right),
\ee
where the precise meaning is the following one:
Let $I(h)$ be a closed interval  containing 0, and $a(h)>0$ such that $a(h)\to 0$ as $h\to 0_+$, and, for all $\varepsilon >0$ there exists $C_\varepsilon >0$ satisfying,
\be
\label{anonexp}
a(h)\geq \frac1{C_\varepsilon} e^{-\varepsilon /h};
\ee
\be
\label{gap1}
\sigma (\widetilde P)\cap \left( (I(h)+[-2a(h), 2a(h)])\backslash I(h)\right) =\emptyset,
\ee
for all $h>0$ small enough. Then, there exists a constant $\varepsilon_1>0$ and a bijection,
$$
\widetilde\beta \, : \, \sigma (\widetilde P)\cap I(h)\,  \to \, {\rm Res} (P)\cap \Gamma (h),
$$
(where we have set,
$\Gamma (h):= (I(h)+[-a(h), a(h)])+i[-\varepsilon_1, 0]$),
such that, for any $\varepsilon >0$, one has,
\be
\label{bij}
\widetilde\beta (\lambda) -\lambda ={\mathcal O}(e^{-(2S_0-\varepsilon)/h}),
\ee
uniformly as $h\to 0_+$.

In particular, since the eigenvalues of $\widetilde P$ are real, one obtains that, for any $\varepsilon >0$, the resonances $\rho$ in $\Gamma (h)$ satisfy,
\be
\label{upperbd}
|\Im \rho |={\mathcal O}(e^{-(2S_0-\varepsilon)/h}).
\ee
Observe that, since $-h^2\Delta +\widetilde V_2$ is elliptic, the eigenvalues of $\widetilde P$ near 0 are actually close (up to $\O(h^2)$) to those of $-h^2\Delta +V_1$. In particular, using Weyl estimates, we see that the number of eigenvalues of $\widetilde P$ inside any small enough fix interval around 0 is $\O(h^{-n})$. Thus, possibly by excluding some particular values of $h$ (e.g. by taking $h$ along a convenient sequence tending to 0), it is in principle possible to construct many such intervals $I(h)$ satisfying \eqref{anonexp}-\eqref{gap1} (see also \cite[Section 2]{HeSj1} and \cite[Section 9]{HeSj2}).

From now on, we consider the particular case where $I(h)$ consists of a unique value, that is, we assume,

{\bf Assumption 3.} {\sl
There exists  $E(h)\in \R$ such that,
\be
\label{gap2}
\begin{aligned}
& E(h)\in \sigma_{disc} (\widetilde P);\\
& E(h)\to 0 \mbox{ as } h\to 0_+;\\
& \sigma(\widetilde P)\cap [E(h)-2a(h), E(h)+2a(h)] =\{ E(h)\}, \\
& \mbox{ where } a(h) \mbox{ satisfies (\ref{anonexp})}.
\end{aligned}
\ee}

Applying \eqref{bij}, we denote by $\rho =\rho (h)$ the unique resonance of $P$ satisfying $\rho - E(h) = \O(e^{-(2S_0-\varepsilon)/h})$ for all $\varepsilon >0$. We also denote by $u_0$ the normalized eigenstate of $\widetilde P$ associated with $E(h)$.

\begin{remark}\sl
By standard results on the tunneling effect (see, e.g., \cite{HeSj1}), it can be shown that the eigenvalues of $\widetilde P$ coincide, up to exponentially small errors, with those of the Dirichlet realization $P_{\mathcal D}$ of $P$ on any domain ${\mathcal D}$ contained in $\widehat I$ and whose interior contains $U$. As a consequence, in Assumption 3 one can equivalently replace $\widetilde P$ with any of such $P_{\mathcal D}$'s.
\end{remark}

As in \cite{Ma1, DaMa}, in order to obtain a lower bound on the width $|\Im\rho|$ of $\rho$, we need to add a further assumption on the size of $u_0$ near some geometric subset of $\partial U$.

{\bf Assumption 4.} \label{Ass6}{\sl
For any $\varepsilon >0$ and for any neighborhood ${\mathcal W}$ of the set $\bigcup_{\gamma \in G}\left( \gamma \cap \partial U \right)$, there exists $C=C(\varepsilon, {\mathcal W})>0$ such that, for all $h>0$ small enough, one has,
$$
\Vert u_0\Vert_{L^2({\mathcal W})\oplus L^2({\mathcal W})} \geq \frac1{C}e^{-\varepsilon /h}.
$$
}

\begin{remark}\sl
\label{remhypMS}
Actually, introducing the microsupport $MS(u_0)$ of $u_0$ as, e.g., in \cite{Ma2} (with obvious changes due to the fact that $u_0$ has two components), in our situation we have $MS(u_0)\cap \{ x\in \partial U\} \subset \{ \xi =0\}$, and  Assumption 4 can be rephrased into,
$$
MS(u_0)\cap \bigcup_{\gamma \in G}\left( \gamma \cap \partial U \right)\times \{ 0\} \not= \emptyset.
$$
\end{remark}
\begin{remark}\sl
By standard results on the propagation of the microsupport (see, e.g., \cite{Ma2}), one can see that Assumption 4 is always satisfied when $n=1$. When $n\geq 2$, a sufficient condition is that for any neighborhood $W$ of $\bigcup_{\gamma \in G}\left( \gamma \cap \partial U \right)\times \{ 0\}$, the set $\bigcup_{t\in\R} \exp tH_{p_1}(W)$ is a neighborhood of $p_1^{-1}(0)\cap \{ x\in U\}$, where $p_1(x,\xi):= \xi^2 + V_1(x)$ and $H_{p_1}:=(\partial_\xi p_1, -\partial_xp_1)$ is the Hamilton field of $p_1$.
\end{remark}

Now, we denote by $G_1\subset G$ be the set of minimal geodesics $\gamma$ such that $MS(u)\cap \left( \gamma \cap \partial U \right)\times \{ 0\} \not= \emptyset$. (By Assumption 4, one has $G_1\not= \emptyset$.) 

 We also denote by $w_0(x,\xi)$ the principal symbol of $W$ and by $w^*_0(x,\xi)$ the principal symbol of $W^*$.

We further assume,

{\bf Assumption 5.} \label{Ass3}{\sl There exists $\gamma\in G_1$ such that,
\begin{itemize}
\item $\gamma$ intersects $\partial \Omega_0$  at a finite number of points $x^{(1)}, \dots, x^{(N)}$  (ordered from the closest to $\partial{\mathcal M}$ up to the closest to $U$);
\item  All intersections between $\gamma$ and $\partial \Omega_0$ are transversal;
\item  $\nabla V_1 \not= \nabla V_2$ on $\gamma\cap \partial \Omega_0$;
\end{itemize}}

{\bf Assumption 6.} {\sl 
 For all $j\in \{ 1, \dots, N\}$, one has $w_0^*(x^{(j)}, i\nabla\varphi (x^{(j)})) \not= 0$ if $j$ is  odd, and $w_0(x^{(j)}, i\nabla\varphi (x^{(j)})) \not= 0$ if $j$ is even.}

Observe that $N$ is necessarily odd. Moreover, the last property is nothing but a condition of ellipticity on the  operators of interaction, similar to that appearing in \cite{GrMa, As}.

Our main result is,

\begin{theorem}\sl
\label{mainth}
Under Assumptions 1 to 6, one has,
$$
\lim_{h\to 0_+} h \ln |\Im \rho (h)| = -2S_0.
$$
\end{theorem}
\begin{remark}\sl
In view of \eqref{upperbd}, this will be implied by the fact that, for any $\varepsilon >0$, there exists $C=C(\varepsilon)>0$ such that,
$$
|\Im \rho (h)| \geq \frac1{C} e^{-(2S_0 +\varepsilon)/h},
$$
for all $h>0$ small enough.
\end{remark}
Concerning the necessity of Assumption 4, we also have the following converse result:
\begin{theorem}\sl
\label{convmainth}
Suppose that Assumptions 1 to 3 are satisfied, but Assumption 4 is not. Assume further that for all $\gamma\in G$, $\gamma$ satisfies the properties listed in Assumption 5. Then, there exists $\delta >0$ such that,
$$
 |\Im \rho (h)| \leq e^{ -2(S_0+\delta)/h}.
$$
\end{theorem}

\section{Preliminaries}\label{Preliminaries}
 In this section, we recall some basic facts on the resonant state $u=(u_1,u_2)$ of $P$ associated with $\rho$, and normalized in such a way that,
 \be
 \label{normal}
 \Vert u\Vert_{L^2(\widehat{\mathcal I})} =1.
 \ee
(From now on, we always write $L^2$ instead of $L^2\oplus L^2$ in order to lighten the notations, and similarly for the Sobolev spaces).

Typically, the properties we are going to recall can be deduced from the same arguments used in the scalar case (see, e.g.,  \cite{HeSj1, HeSj2, HiSi}), and we refer the reader interested in more details to, e.g., \cite{Kl, GrMa}.

First of all, by Lithner-Agmon estimates (as, e.g., in \cite{Kl}), we have that, for any $\varepsilon >0$ and any bounded set ${\mathcal B}\subset \R^n$,
\be
\label{agmon}
\Vert e^{\varphi /h }u\Vert_{H^1({\mathcal B})} =\O(e^{\varepsilon /h}),
\ee
where $\varphi$ is defined in \eqref{agmondist}. In addition, with the same techniques used in \cite{HeSj2}, it can be seen that near the well $U$, $u$ is exponentially close to $u_0$, in the sense that there exists $\delta >0$ and a neighborhood $\Omega_U$ of $\overline U$ such that,
$$
\Vert u-u_0\Vert_{H^1(\Omega_U)} = \O(e^{-\delta /h}).
$$
in particular, the property of $u_0$ in Assumption 4 extends to $u$, too.

As in \cite{HeSj2, DaMa}, we consider the set (called the set of  ``points of type 1'' in \cite{HeSj2}),
$$
{\mathcal T}_1 := \bigcup_{\gamma \in G}(\gamma \cap \partial {\mathcal M}).
$$
We also set $p_2(x,\xi)=\xi^2 +V_2(x)$ and, by the same arguments as in \cite[Section 9]{HeSj2} (but adapted to our case of a system), we see that if a bounded subset $\mathcal B$ of $\mathcal M$ stays away from the $x$-projection of the set,
$$
\widetilde{\mathcal T}_1:= \bigcup_{t\in\R} \exp tH_{p_2}({\mathcal T}_1\times\{0\}),
$$
then there exists $\delta >0$ such that,
\be
\label{type2}
\Vert u \Vert_{H^1({\mathcal B})} =\O(e^{-(S_0+\delta)/h}).
\ee

Now, let $\Omega_1$ be a smooth bounded open domain containing the closure of the island $\widehat{\mathcal I}$, and write the interaction $W$ as,
\be
\label{interW}
W= r_0(x) + hr_1(x)\cdot \nabla_x,
\ee
where $r_0$ is complex-valued, and $r_1$ is (complex) vector-valued.
 By the Stokes formula on $\Omega_1$, we have,
\be
\label{formHS}
(\Im \rho)\Vert u\Vert^2_{L^2(\Omega_1)} = -h^2\Im  \int_{\partial\Omega_1} \frac{\partial u}{\partial\nu}\cdot \overline{u}ds + h^2\Im\int_{\partial\Omega_1}(r_1\cdot \nu)u_2\overline{u_1}ds,
\ee
where $ds$ is the surface measure on $\partial\Omega_1$, and $\nu$ stands for the outward ponting unit normal to $\Omega_1$.\\
(Note that, if $W$ had been a more general pseudodifferential operator, the previous formula would not have been valid anymore, but could have been replaced by another one involving the multiplication by a cut-off function instead of the restriction to $\Omega_1$.)

Using \eqref{agmon}-\eqref{type2} and \eqref{normal}, we easily deduce the existence of some $\delta >0$ such that,
\be
\label{formimrho}
\begin{aligned}
\Im \rho = -h^2\Im  \int_{\partial\Omega_1\cap {\mathcal B}} \frac{\partial u}{\partial\nu}\cdot \overline{u}ds + h^2\Im\int_{\partial\Omega_1\cap {\mathcal B}}&(r_1\cdot \nu)u_2\overline{u_1}ds\\
&+\O(e^{-(2S_0+\delta)/h}),
\end{aligned}
\ee
where ${\mathcal B}$ is an arbitrarily small neighborhood of the $x$-projection $\Pi_x\widetilde{\mathcal T}_1$ of $\widetilde{\mathcal T}_1$.

Now, we fix some arbitrary $z_1\in\Pi_x{\mathcal T}_1$, and we denote by $Z_1$ a small enough neighborhood of $z_1$ in $\partial M$. For $t_0>0$ sufficiently small, we set,
$$
\Lambda_\pm := \bigcup_{0<\pm t<2t_0} \exp tH_{p_2}(Z_1\times\{0\}).
$$
Then $\Lambda_\pm \subset \{ p_2=0\}$ and, since $Z_1\times\{0\}$ is isotropic, $\Lambda_\pm$ is Lagrangian (see, e.g., \cite{Ma2}). Since $\nabla V_2(z_1)\not= 0$ and $H_{p_2} =(2\xi, -\nabla V_2)$, it is also easy to check that both $\Lambda_+$ and $\Lambda_-$ project bijectively on the base. Since in addition $p_2$ is an even function of $\xi$, we finally obtain (e.g., as in \cite[Section 5]{DaMa}) the existence of a real-analytic function $\psi$, defined on the $x$-projection of $\Lambda_\pm$, such that,
\be
\begin{aligned}
& \Lambda_\pm \, : \, \xi =\pm \nabla\psi (x);\\
& (\nabla\psi (x))^2 + V_2(x) =0.
\end{aligned}
\ee
Setting $z_0:=\Pi_x (\exp t_0H_{p_2}(z_1, 0))$, and still denoting by $\psi$ an holomorphic extension of $\psi$ to a complex neighborhood of $z_0$, one can prove as in \cite[Proposition 5.1]{DaMa},

\begin{proposition}\sl
\label{Hphi}
For any $\varepsilon_1>0$, one has,
$$
e^{-i\psi /h+S_0/h}u\in H_{ \varepsilon_1|\Im x|, z_0},
$$
where $H_{\varepsilon_1|\Im x|, z_0}$ is the Sj\"ostrand's space consisting of $h$-dependent holomorphic functions $v=v(x;h)$ defined in a complex neighborhood of $z_0$, such that, for all $\varepsilon >0$,
$$
v(x,h)={\mathcal O}(e^{(\varepsilon_1|\Im x|+\varepsilon)/h}),
$$
uniformly for $x\in \C^n$ close enough to $z_0$ and $h>0$ small enough.
\end{proposition}
\begin{remark}\sl Obviously, the result of this proposition can be re-written as,
$$
u\in H_{-\Im \psi -S_0 + \varepsilon_1|\Im x|, z_0}.
$$
\end{remark}
\begin{proof} The proof is very similar to that of \cite[Proposition 5.1]{DaMa}, with the only difference that here we have to deal with a matrix-operator, instead of a scalar one. For the sake of completeness, we outline the main steps. At first, for $x$ close enough to $z_0$, we write $v(x):= e^{-i\psi (x)/h+S_0/h}u(x)$ as the oscillating integral,
\be
\label{repv}
v(x)= (2\pi)^{-n}\int e^{i|\xi|\theta (x,y,\frac{\xi}{|\xi|})}a(x-y,\frac{\xi}{|\xi|})v(y)\chi(y)dyd\xi,
\ee
where $\chi$ is a cut-off function around $z_0$, $\theta (x,y,\tau):=(x-y)\tau+\frac{1}{2}i(x-y)^2$,  and $a(x,\tau):=1+\frac{1}{2}ix\tau$.
Thanks to the standard ellipticity of $P$, for $|\xi| \geq Ch^{-1}$ (with $C>0$ a large enough constant), one can construct a $\C^2$-valued analytic symbol $b=b(x,y,\tau, \xi, h)\sim \sum_{k\geq 0} b_k(x,y,\tau, h)|\xi|^{-k}$ such that,
$$
e^{-i|\xi|\theta (x,y,\tau)} Q(y,hD_y)\left(e^{i|\xi|\theta (x,y,\tau)}b\right) =a(x-y,\tau) +\O(e^{-\delta |\xi|}),
$$
with $\delta >0$, and where $Q(y,hD_y):={}^t\left(e^{-i\psi (y)/h}P(y,hD_y)e^{i\psi (y)/h}-\rho\right)$. Inserting this estimate into \eqref{repv} and using that, for all $\varepsilon >0$, one has $v=\O(e^{\varepsilon /h})$ uniformly on the real near $z_0$, together with the fact that $Qv=0$, we obtain,
\be
\label{reducrepv}
v(x) =(2\pi)^{-n}\int_{|\xi|\leq C/h} e^{i|\xi|\theta (x,y,\frac{\xi}{|\xi|})}a(x-y,\frac{\xi}{|\xi|})v(y)\chi(y)dyd\xi + \O(e^{-\delta'/h}),
\ee
for some $\delta'>0$. 

Then, splitting the remaining integral into $\int_{\{|\xi|\leq\frac{\varepsilon_1}{h}\}}$ and $\int_{\{\frac{\varepsilon_1}{h}\leq|\xi|\leq\frac{C}{h}\}}$, we immediately observe that  the first  term $\int_{\{|\xi|\leq\frac{\varepsilon_1}{h}\}}$ is $\O(e^{(\varepsilon_1|\Im x| +\varepsilon)/h})$ for all $\varepsilon >0$. Finally, using that $u$ is outgoing and the results of \cite[Section 9]{HeSj2}, we see that $MS(e^{S_0/h}u)\cap \Lambda_-=\emptyset$. Therefore, by propagation of $MS$, and since also $MS(e^{S_0/h}u)\cap \left(\partial M\times \R^n\right) \subset \{\xi =0\}$, we deduce that, above a neighborhood of $z_0$, one necessarily has $MS(e^{S_0/h}u)\subset \Lambda_+$, and thus $MS(v)\subset \{ \xi =0\}$. But, after the change of variables $\xi \mapsto \xi /h$, this exactly means that the term $\int_{\{\frac{\varepsilon_1}{h}\leq|\xi|\leq\frac{C}{h}\}}$ is exponentially small as $h\to 0_+$.
\end{proof}

Thanks to Proposition \ref{Hphi}, we can enter the framework of analytic pseudodifferential calculus of \cite[Sections 4-5]{Sj}. In particular, working in the Sj\"ostrand's space 
$H_{-\Im \psi -S_0 + \varepsilon_1|\Im x|, z_0}$, we can represent 
$Pu$ as,
\be
\label{pseudo}
Pu(x)=\frac1{(2\pi h)^n}\int_{\Gamma(x)}e^{i(x-y)\xi /h -[(x-\alpha_x)^2+(y-\alpha_x)^2]/2h}p (\alpha_x,\xi;h) v(y)dyd\xi d\alpha_x,
\ee
where $p$ is a matrix-valued analytic symbol that satisfies,
$$
p (\alpha_x,\xi;h)=\left(\begin{array}{cc}
\xi^2 + V_1(\alpha_x) & 0\\
0 & \xi^2 + V_2(\alpha_x)
\end{array}\right) +\O(h),
$$
and $\Gamma (x)$ is the (bounded, singular) contour of $\C^{3n}$ defined by,
$$
\Gamma (x)\, :\, \left\{
\begin{aligned}
& \xi = \nabla\psi (\alpha_x) +2i\varepsilon_1\frac{\overline{x-y}}{|x-y|}\, ;\\
& |x-y| \leq r,\,\, y\in \C^n\,\, (r \mbox{ small enough with  respect to } \varepsilon_1)\, ;\\
& |x-\alpha_x| \leq r,\, \alpha_x\in\R^n.
\end{aligned}
\right.
$$
Now, as in \cite{DaMa}, we take local coordinates $(x',x_n)\in \R^{n-1}\times \R$ centered at $z_1$, in such a way that $dV_2(z_1)\cdot x =-cx_n$ with $c>0$. Then, taking advantage of the fact that $hD_{x_n}$ and  $h^2D_{x_n}^2$  can be represented as in \eqref{pseudo} with $p$ substituted with $\xi_n$ and $\xi_n^2$, respectively, we see that $p (\alpha_x, \xi;h)$ can actually be written as,
$$
p (\alpha_x,\xi;h)=\left(\begin{array}{cc}
\xi_n^2 + a_1 & hb_1+hb_2\xi_n\\
hb_3+hb_4 \xi_n & \xi_n^2 + a_2
\end{array}\right) ,
$$
where the $a_j$'s and the $b_j$'s are functions of $(\alpha_x, \xi' ;h)$ only (not of $\xi_n$), and satisfy,
$$
\begin{aligned}
& a_j(\alpha_x, \xi' ;h)=(\xi')^2+V_j(\alpha_x)+\O(h)\quad &(j=1,2);\\
& b_k(\alpha_x, \xi' ;h)=\O(1)\quad \mbox{ on } \Gamma(x) &(k=1,2,3,4).
\end{aligned}
$$
Now, on $\Gamma(x)$, we have that $\xi$ remains close to $\nabla\psi (\alpha_x)$, which in turn remains close to $(0, \sqrt{c\alpha_{x_n}})$, with $\alpha_{x_n}$ close to $\delta_0:=z_{0,n}>0$. In particular, $a_1$ remains close to $V_1(z_0)>0$, and $a_2$ remains close to $V_2(z_0)<0$. 

Then, re-writing $p (\alpha_x,\xi;h)$ as,
$$
p (\alpha_x,\xi;h)=\left( \xi_n {\mathbf I}_2 + \frac{h}2{\mathbf J}\right)^2 + A
$$
with,
$$
\begin{aligned}
& {\mathbf J}:=  \left(\begin{array}{cc}
0 & b_2\\
b_4  & 0
\end{array}\right);\\
& A:= \left(\begin{array}{cc}
a_1 & hb_1\\
hb_3 & a_2
\end{array}\right) -\frac{h^2}4 {\mathbf J}^2,
\end{aligned}
$$
and using the analytic symbolic calculus of \cite{Sj}, we see that, as an operator on $H_{-\Im \psi -S_0 + \varepsilon_1|\Im x|, z_0}$, $P-\rho$ can be factorized into,
$$
P-\rho = ( hD_{x_n} {\mathbf I}_2+ B_+)(hD_{x_n}{\mathbf I}_2+B_-),
$$
where the symbol $\beta_\pm$ of $B_\pm$ does not depend on $\xi_n$, and satisfies,
$$
\beta_\pm (\alpha_x, \xi' ;h)=\pm  \left(\begin{array}{cc}
i\sqrt{a_1-\rho} & 0\\
0& \sqrt{\rho-a_2}
\end{array}\right) + \O(h).
$$
(Here, $\sqrt{\cdot}$ stand for the principal square-root of complex numbers with positive real part.)

In particular, $\xi_n{\mathbf I}_2 + \beta_+$ remains elliptic on $\Gamma (x)$, and thus, by applying an analytic parametrix of $ hD_{x_n} {\mathbf I}_2+ B_+$, the equation $Pu =\rho u$ becomes,
\be
\label{derivnorm1}
(hD_{x_n}+B_- )u=0\quad \mbox{in } H_{-\Im \psi -S_0 + \varepsilon_1|\Im x|, z_0}.
\ee
In the same way, $P-\rho$ can also be factorized into,
$$
P-\rho = ( hD_{x_n} {\mathbf I}_2+ \widetilde B_+)(hD_{x_n}{\mathbf I}_2+\widetilde B_-),
$$
where the symbol $\widetilde\beta_\pm$ of $ \widetilde B_\pm$  satisfies,
$$
 \widetilde \beta_\pm (\alpha_x, \xi' ;h)=\pm  \left(\begin{array}{cc}
-i\sqrt{a_1-\rho} & 0\\
0& \sqrt{\rho-a_2}
\end{array}\right) + \O(h).
$$
Since $\xi_n{\mathbf I}_2 + \widetilde \beta_+$ remains elliptic, too,  on $\Gamma (x)$, this also leads to,
\be
\label{derivnorm2}
(hD_{x_n}+\widetilde B_- )u=0\quad \mbox{in } H_{-\Im \psi -S_0 + \varepsilon_1|\Im x|, z_0}.
\ee

At that point, we can proceed as in \cite{Ma2, DaMa} (or, also, \cite[Section 10]{HeSj2}), and, by taking a realization on the real domain  of $B_-$ and $\widetilde B_-$ near $z_0$ (see \cite{Sj}), we conclude to the existence of four analytic pseudodifferential operators $Q_j$ and $\widetilde Q_j$ ($j=1,2$) acting on $L^2(\{ x_n=\delta_0\})$, such that, near $z_0$,
\be
\label{derivnorm1}
\begin{aligned}
& \left. h\frac{\partial u_1}{\partial x_n}\right|_{x_n=\delta_0}=Q_1 u_1 + \O(h)\Vert \chi u_2\Vert_{\{ x_n=\delta_0\}}=\widetilde Q_1 u_1 + \O(h)\Vert \chi u_2\Vert_{\{ x_n=\delta_0\}};\\
& \left. h\frac{\partial u_2}{\partial x_n}\right|_{x_n=\delta_0}=Q_2 u_2 + \O(h)\Vert \chi u_1\Vert_{\{ x_n=\delta_0\}}=\widetilde Q_2 u_2 + \O(h)\Vert \chi u_1\Vert_{\{ x_n=\delta_0\}},
\end{aligned}
\ee
where $\chi$ is some cut-off function that localizes near $z_0$, and the symbols $q_j$ and $\widetilde q_j$ of $Q_j$ and $\widetilde Q_j$  satisfy,
\be
\label{qj}
\begin{aligned}
& q_1(x', \xi'; h) = -\sqrt{(\xi')^2+V_1(x', \delta_0)-\rho} +\O(h);\\
& \widetilde q_1(x', \xi'; h) = \sqrt{(\xi')^2+V_1(x', \delta_0)-\rho} +\O(h);\\
& q_2(x', \xi'; h) = i\sqrt{\rho -(\xi')^2-V_2(x', \delta_0)} +\O(h);\\
&\widetilde  q_2(x', \xi'; h) = i\sqrt{\rho -(\xi')^2-V_2(x', \delta_0)} +\O(h).
\end{aligned}
\ee
More precisely (see \cite[Section 5]{Sj}), these operators (say, $Q_j$) are of the type,
\be
\label{factQj}
\begin{aligned}
Q_j v(x';h)=\frac1{(2\pi h)^{3(n-1)/2}}\int_{\R^{3(n-1)}}&e^{i(x'-y')\xi'/h - [(x'-\alpha_{x'})^2+(y'-\alpha_{x'})^2]/h}\\
&\times q_j(x',\xi';h)\chi_0 (x',y',\alpha_{x'},\xi')v(y')dy'd\alpha_{x'}d\xi',
\end{aligned}
\ee
where $\chi_0\in C_0^\infty (\R^{4(n-1)})$ is a cut-off function on a small enough neighborhood of the point $(z_0',z_0',z_0', \nabla_{x'}\psi (z_0))$.

(Note that, even for such realizations on the real of analytic pseudodifferential operators, $\xi'$ remains close to 0.)

Here we observe the important fact that $Q_1 +\widetilde Q_1$ is $\O(h)$, a consequence of which is that \eqref{derivnorm1} actually implies,
\be
\label{derivnorm}
\begin{aligned}
& \left. h\frac{\partial u_1}{\partial x_n}\right|_{x_n=\delta_0}= \O(h)\Vert \chi u_1\Vert_{\{ x_n=\delta_0\}}+ \O(h)\Vert \chi u_2\Vert_{\{ x_n=\delta_0\}};\\
& \left. h\frac{\partial u_2}{\partial x_n}\right|_{x_n=\delta_0}=Q_2 u_2 + \O(h)\Vert \chi u_1\Vert_{\{ x_n=\delta_0\}}.
\end{aligned}
\ee

Now, denoting by  $C_2$ the pseudodifferential operators of the same form as $Q_2$, with  symbol  $c_2:= [\Re\rho -(\xi')^2-V_2(x', \delta_0)]^{1/4}$, the symbolic calculus of \cite{Sj} gives us,
$$
 Q_2 = iC_2^* C_2 +\O(h).
$$
Turning back to formula \eqref{formimrho}, we can treat in that way all the points of $\partial\Omega_1\cap {\mathcal B}$, and, using a convenient partition of unity $(\chi_j^2)_{1\leq j \leq N}$ on this set, we can write,
$$
\begin{aligned}
\Im \rho & = -h^2\sum_{j}\Im  \la \chi_j\frac{\partial u}{\partial\nu},\chi_j u\ra_{L^2(\partial\Omega_1)} + \O(h^2)\Vert u\Vert^2_{L^2(\partial \Omega_1)}+\O(e^{-(2S_0+\delta)/h})\\
&=  -h\sum_{j}\Im  \la i\chi_jC_{2,j}u_2, C_{2,j}\chi_j u_2\ra_{L^2(\partial\Omega_1)}\\
&\hskip 1cm+\O(h^2)\Vert u\Vert^2_{L^2(\partial \Omega_1)}+\O(e^{-(2S_0+\delta)/h})\\
& =-h\sum_{j}\Vert C_{2,j}\chi_ju_2\Vert^2_{L^2(\partial\Omega_1)}+ \O(h^2)\Vert u\Vert^2_{L^2(\partial \Omega_1)}+\O(e^{-(2S_0+\delta)/h}),
\end{aligned}
$$
where the $C_{2,j}$'s  ($j=1,\dots ,N$) are the pseudodifferential operators constructed as before on  the support of $\chi_j$, and where we have used \eqref{derivnorm} and the fact that the commutator $[\chi_j, C_{2,j}]$ is $\O(h)$.

Using the ellipticity of the symbol of $C_{2,j}$, we deduce the existence of a constant $C_0>0$ such that,
$$
|\Im \rho |\geq \frac{h}{C_0}\sum_j\Vert \chi_ju_2\Vert^2_{L^2(\partial\Omega_1)}-C_0h^2\Vert u\Vert^2_{L^2(\partial \Omega_1)}-C_0e^{-(2S_0+\delta)/h},
$$
and thus, since $\sum_j\chi_j^2 =1$ on $\partial \Omega_1\cap {\mathcal B}$, while  $u=\O(e^{-(S_0+\delta')/h})$ (with $\delta' >0$) on $\partial \Omega_1\backslash {\mathcal B}$, we have proved,
\begin{proposition}\sl
\label{estrhou}
There exist two positive constants $C_0$ and $\delta$, such that,
$$
|\Im \rho |\geq \frac{h}{C_0}\Vert u_2\Vert^2_{L^2(\partial\Omega_1)}-C_0h^2\Vert u_1\Vert^2_{L^2(\partial \Omega_1)}-C_0e^{-(2S_0+\delta)/h},
$$
for all $h>0$ small enough.
\end{proposition}

\section{Reductio ad absurdum}

From now on, we proceed by contradiction, assuming the existence of some constant $\delta_1>0$ such that,
\be
\label{absurd}
|\Im \rho | = \O (e^{-(2S_0 + \delta_1)/h}),
\ee
uniformly as $h\to 0_+$ (possibly along some sequence).

By Proposition \ref{estrhou}, this implies,
\be
\label{estu2-u1partial}
\Vert u_2\Vert^2_{L^2(\partial\Omega_1)}=\O (h\Vert u_1\Vert^2_{L^2(\partial \Omega_1)}+e^{-(2S_0+\delta')/h}),
\ee
for some positive constant $\delta$. Here, let us observe that the dependence with respect to $\Omega_1$ of the constants $C_0$ and $\delta$ that appear in Proposition \ref{estrhou} is only related to the ellipticity of $\xi_n +\sqrt{(\xi')^2 + V_2(x)}$ on $\{(x, \nabla\psi (x))\, ; \, x\in \partial\Omega_1\}$. In particular, they can be taken uniformly as long as the distance between $\partial\Omega_1$ and $\{V_2 =0\}$ remains larger than some fixed positive constant. Therefore, choosing a convenient 1-parameter family of such open sets $\Omega_1$ (for instance $\Omega_{1,t}:=\{ V_2(x)= -t\}$ with $t\in [t_0,t_1]$, $0<t_0<t_1$ small enough), and integrating with respect to this parameter, we obtain from \eqref{estu2-u1partial},
\be
\label{estu2-u1}
\Vert u_2\Vert^2_{L^2({\mathcal A})}=\O (h\Vert u_1\Vert^2_{L^2({\mathcal A})}+e^{-(2S_0+\delta')/h}),
\ee
where now ${\mathcal A}$ is a topological annulus surrounding $\widehat{\mathcal I}$, e.g. of the type,
\be
\label{Annulus}
{\mathcal A}=\{ -t_1 < V_2(x)<-t_0, x\in \Omega_1'\},
\ee
where $\Omega_1'$ stands for some fixed small enough neighborhood of $\widehat{\mathcal I}$.
Then, using the equation $Pu =\rho u$, we have,
$$
\la (-h^2\Delta + V_1 -\rho)u_1, u_1\ra_{L^2({\mathcal A})} = -h\la Wu_2, u_1\ra_{L^2({\mathcal A})},
$$
and thus, by Stokes formula,
$$
\begin{aligned}
\Vert h\nabla u_1\Vert^2_{\mathcal A} + & \la ( V_1 -\rho)u_1, u_1\ra_{\mathcal A}-h^2\la \partial_\nu u_1, u_1\ra_{\partial\mathcal A}\\
& =\O(h)\left(\Vert u_1\Vert^2_{\mathcal A}+ \Vert u_2\Vert^2_{\mathcal A}+h\Vert u_1\Vert^2_{\partial\mathcal A}+h \Vert u_2\Vert^2_{\partial\mathcal A}+\Vert h\nabla u_1\Vert^2_{\mathcal A} \right),
\end{aligned}
$$
where $\partial_\nu$ is the outward pointing normal derivative on $\partial A$. Taking the real part, and using that $\Re (V_1-\rho)$ is positive near $\mathcal A$, together with the fact that, thanks to \eqref{derivnorm}, one has  $\la \partial_\nu u_1, u_1\ra_{\partial\mathcal A}=\O (\Vert u\Vert^2_{\partial\mathcal A})$, we deduce (for $h$ sufficiently small),
$$
\Vert h\nabla u_1\Vert^2_{\mathcal A}+\Vert u_1\Vert^2_{\mathcal A}=\O (h\Vert u_2\Vert^2_{\mathcal A}+h^2\Vert u\Vert^2_{\partial\mathcal A}).
$$
Now, using both \eqref{estu2-u1} and \eqref{estu2-u1partial} (with $\partial{\mathcal A}$ instead of $\partial\Omega_1$), this gives us,
$$
\Vert h\nabla u_1\Vert^2_{\mathcal A}+\Vert u_1\Vert^2_{\mathcal A}=\O (h^2\Vert u_1\Vert^2_{\mathcal A}+h^2\Vert u_1\Vert^2_{\partial\mathcal A}+e^{-(2S_0+\delta')/h}),
$$
that is, for $h$ sufficiently small,
\be
\label{estu1u1bord}
h^2\Vert \nabla u_1\Vert^2_{\mathcal A}+\Vert u_1\Vert^2_{\mathcal A}=\O (h^2\Vert u_1\Vert^2_{\partial\mathcal A}+e^{-(2S_0+\delta')/h}).
\ee
At this point, we can use the standard Sobolev estimate, 
$$
\Vert u_1\Vert^2_{\partial\mathcal A}=\O(\Vert u_1\Vert^2_{\mathcal A}+ \Vert u_1\Vert_{\mathcal A}\Vert \nabla u_1\Vert_{\mathcal A}),
$$
that implies,
$$
h^2\Vert u_1\Vert^2_{\partial\mathcal A}=\O( h\Vert u_1\Vert^2_{\mathcal A}+ h^3\Vert \nabla u_1\Vert^2_{\mathcal A}),
$$
and thus, once inserted into \eqref{estu1u1bord}, permits us to conclude that, for $h$ small enough,
$$
\Vert h\nabla u_1\Vert^2_{\mathcal A}+\Vert u_1\Vert^2_{\mathcal A}=\O (e^{-(2S_0+\delta')/h}).
$$
Summing up, and gathering with \eqref{estu2-u1}, we have proved,
\begin{proposition}\sl
\label{propestabs}
Under Assumption \eqref{absurd}, and with ${\mathcal A}$ given in \eqref{Annulus}, there exists $\delta_1 >0$ such that, 
$$
\Vert u\Vert^2_{L^2({\mathcal A})}=\O (e^{-(S_0+\delta_1)/h}).
$$
\end{proposition}
The purpose of the next sections will be to propagate this smallness of $u$ across the shore $\partial{\mathcal M}$, the crest $\partial\Omega_0$, and up to the interior of the well $U$, in such a way that one finally gets a contradiction with Assumption 4.

\section{Propagation across the shore $\partial{\mathcal M}$}

In order to propagate the smallness of $u$ across $\partial{\mathcal M}$, we adopt the same strategy as in \cite{DaMa}, that is, we start by establishing global Carleman estimates around $\partial{\mathcal M}$.

We fix some  $\mu_0 >0$  sufficiently small, and, for $\mu\in (0,\mu_0]$, we consider the neighborhood ${\mathcal N}_\mu$ of $\partial{\mathcal M}$ given by,
\be
\label{Nmu}
{\mathcal N}_\mu:=\{ -\mu_0 \leq V_2(x)\leq \mu, \, x\in \Omega_1\},
\ee
(where, as before, $\Omega_1$ is some fixed small enough neighborhood of $\widehat{\mathcal I}$).

We also set,
$$
\begin{aligned}
& \Sigma := {\mathcal N}\cap \{ V_2=\mu\};\\
& \Sigma_0 := {\mathcal N}\cap \{ V_2=-\mu_0\}.
\end{aligned}
$$
By the same geometrical considerations as in \cite[Section 6]{DaMa}, we see that, on $\Sigma$, the function $\varphi$ satisfies,
$$
\varphi\left|_{\Sigma\mu}\right.  \geq S_0 - c_0\mu ^{3/2},
$$
where the constant $c_0>0$ does not depend on $\mu$. Therefore, using \eqref{agmon} and Sobolev estimates, we deduce that, for any $\mu, \varepsilon >0$ small enough, we have,
\be
\label{estbordN1}
u=\O(e^{-(S_0-c_0\mu^{3/2}-\varepsilon)/h})\quad \mbox{uniformly on } {\mathcal N}_\mu,
\ee
and the same holds for all the derivatives of $u$.

In addition, since $\Sigma_0$ stays away from $\widehat{\mathcal I}$, by Proposition \ref{propestabs} (plus standard Sobolev estimates), we know the existence of some $\delta_0>0$ constant such that,
\be
\label{estbordN2}
\Vert u\Vert_{H^2(\Sigma_0)} =\O(e^{-(S_0+\delta_0)/h}).
\ee
We plan to extend this estimate up to $\Sigma$. We set,
\be
\label{defv}
v(x):= e^{\alpha (\mu - V_2(x))/h} u(x),
\ee
where $\alpha >0$ satisfies,
$$
2\alpha \mu_0 \leq \delta_0.
$$
Then, by \eqref{estbordN1}, for any $\varepsilon >0$, $v$ satisfies,
\be
\label{estbordN1v}
\Vert v\Vert_{H^2(\Sigma)}=\O(e^{-(S_0-c_0\mu^{3/2}-\varepsilon)/h}),
\ee
and, by \eqref{estbordN2},
\be
\label{estbordN2v}
\Vert v\Vert_{H^2(\Sigma_0)} =\O(e^{-S_0/h}).
\ee

 Moreover, $v$ is solution to,
$$
(A+iB)v=0,
$$
with,
$$
\begin{aligned}
& A:= P-\Re\rho -\alpha^2(\nabla V_2)^2;\\
& B:=-2h\alpha(\nabla V_2)\cdot D_x +ih\alpha\Delta V_2-\Im\rho-ih\alpha\left(\begin{array}{cc}
0 & - r_1\cdot \nabla V_2\\
\overline{r_1}\cdot \nabla V_2 & 0
\end{array}\right),
\end{aligned}
$$
where we have used \eqref{interW}. In particular, $A$ and $B$ are formally selfadjoint, and we have,
\be
\label{A+iB}
0=\Vert (A+iB)v\Vert^2_{L^2({\mathcal N}_\mu)} = \Vert Av\Vert^2_{L^2({\mathcal N}_\mu)} +\Vert Bv\Vert^2_{L^2({\mathcal N}_\mu)} +2\Im \la Av, Bv\ra_{L^2({\mathcal N}_\mu)}.
\ee
As in \cite{DaMa}, the key-point is the following Carleman estimate:
\begin{lemma}\sl 
\label{carleman}
If $\alpha$ and $\mu_0$ are chosen sufficiently small, there exists a constant $C>0$ and, for all $\varepsilon >0$, a constant $C_\varepsilon >0$, such that,
$$
\begin{aligned}
\Im \la Av, Bv\ra_{L^2({\mathcal N}_\mu)}\geq  \frac{h}{C}\Vert v_2\Vert^2_{L^2({\mathcal N}_\mu)}-Ch\Vert v_1&\Vert^2_{L^2({\mathcal N}_\mu)} -Ch\Vert Av\Vert^2_{L^2({\mathcal N}_\mu)}\\ 
&-C_\varepsilon e^{-2(S_0-c_0\mu^{3/2}-\varepsilon)/h},
\end{aligned}
$$
uniformly for $h>0$ small enough.
\end{lemma}
\begin{proof} Using Green's formula, together with \eqref{estbordN1v}-\eqref{estbordN2v} and the fact that $\partial{\mathcal N}_\mu = \Sigma_0\cup\Sigma$, we obtain,
\be
\label{imAB}
\Im \la Av, Bv\ra_{L^2({\mathcal N}_\mu)}=\frac{i}2\la [A,B]v,v\ra_{L^2({\mathcal N}_\mu)}+\O(e^{-2(S_0-c_0\mu^{3/2}-\varepsilon)/h}),
\ee
uniformly with respect to $h$, and with $\varepsilon, \mu$ arbitrarily small. Moreover, setting
$$\begin{aligned}
& {\mathbf V}:=\left(\begin{array}{cc}
V_1 &0\\
0 & V_2
\end{array}\right),\, 
{\mathbf W}:=\left(\begin{array}{cc}
0 & W\\
W^* & 0
\end{array}\right), \, 
{\mathbf R}:=\left(\begin{array}{cc}
0 & - r_1\cdot \nabla V_2\\
\overline{r_1}\cdot \nabla V_2 & 0
\end{array}\right),\\
& A_0:= -h^2\Delta + {\mathbf V} -\Re\rho -\alpha^2(\nabla V_2)^2=A-h{\mathbf W},\\
&  B_0:= -2h\alpha(\nabla V_2)\cdot D_x +ih\alpha\Delta V_2-\Im\rho=B+ih\alpha {\mathbf R},
\end{aligned}
$$
we have,
\be
\label{formAB}
\frac{i}2[A,B]=\frac{i}2[A_0,B_0]+h[{\mathbf W},B]-ih\alpha[A_0,{\mathbf R}],
\ee
and, as in \cite[Section 6]{DaMa}, a straightforward computation leads to,
\be
\label{A0B0}
\frac{i}2[A_0,B_0]=\alpha h^3\left(Q_2+2Q_1+\frac12(\Delta^2V_2)\right)) + \alpha h \nabla V_2\cdot \nabla{\mathbf V}-\alpha^3 h Q_0
\ee
with,
$$
\begin{aligned}
& Q_2:= 2\sum_{j,k}(\partial_j\partial_kV_2)\partial_j\partial_k;\\
&Q_1:= \nabla(\Delta V_2)\cdot \nabla ;\\
&Q_0:= (\nabla V_2)\cdot(\nabla (\Delta V_2)^2).
\end{aligned}
$$
Here, we first observe,
$$
\la  (\nabla V_2\cdot \nabla{\mathbf V})v,v\ra_{{\mathcal N}_\mu }= \la (\nabla V_2\cdot \nabla V_1) v_1,v_1\ra + \Vert (\nabla V_2)v_2\Vert^2_{{\mathcal N}_\mu }
$$
and therefore, since $\nabla V_2$ never vanishes on $\partial {\mathcal M}$, there exists a constant $C_0>0$ (depending only on the geometry of $V_1$ and $V_2$ near $\partial{\mathcal M}$), such that,
\be
\label{estgrad}
\la  (\nabla V_2\cdot \nabla{\mathbf V})v,v\ra_{{\mathcal N}_\mu }\geq \frac1{C_0}\Vert v_2\Vert^2_{{\mathcal N}_\mu }-C_0\Vert v_1\Vert^2_{{\mathcal N}_\mu }.
\ee

Then, by doing an integration by parts and by using \eqref{estbordN1v}-\eqref{estbordN2v} we first see,
\be
\label{estA}
\Vert h\nabla v\Vert_{{\mathcal N}_\mu }^2 =\O(\Vert Av\Vert_{{\mathcal N}_\mu }^2+\Vert v\Vert_{{\mathcal N}_\mu }^2+e^{-2(S_0-c_0\mu^{3/2}-\varepsilon)/h}).
\ee
and then,
$$
\begin{aligned}
h^2 \la Q_2 v,v\ra_{{\mathcal N}_\mu } & =\O (\Vert h\nabla v\Vert_{{\mathcal N}_\mu }^2 + h^2\Vert v\Vert_{{\mathcal N}_\mu }^2+e^{-2(S_0-c_0\mu^{3/2}-\varepsilon)/h})\\
& =\O (|\la -h^2\Delta v, v\ra_{{\mathcal N}_\mu }|+ h^2\Vert v\Vert_{{\mathcal N}_\mu }^2+e^{-2(S_0-c_0\mu^{3/2}-\varepsilon)/h}).
\end{aligned}
$$
Therefore, for any constant $C\geq 1$ arbitrarily large, and $h>0$ sufficiently small,
$$
h^2 |\la Q_2 v,v\ra_{{\mathcal N}_\mu } | \leq C\Vert h^2\Delta v\Vert_{{\mathcal N}_\mu }^2+\frac2{C}\Vert v\Vert_{{\mathcal N}_\mu }^2+C_\varepsilon e^{-2(S_0-c_0\mu^{3/2}-\varepsilon)/h}
$$
(where $C_\varepsilon$ depends on $\varepsilon$ only, which in turn is arbitrarily small). We deduce (with some new constant $C'>0$),
\be
\label{estQ2}
\begin{aligned}
h^2 |\la Q_2 v,v\ra_{{\mathcal N}_\mu } | \leq & 2C\Vert A v\Vert_{{\mathcal N}_\mu }^2+2C\Vert ({\mathbf V}-\Re\rho -\alpha^2(\nabla V_2)^2)v\Vert_{{\mathcal N}_\mu }^2+C'h\Vert  v\Vert_{{\mathcal N}_\mu }^2\\
& +C'h\Vert h\nabla v\Vert_{{\mathcal N}_\mu }^2+\frac2{C}\Vert v\Vert_{{\mathcal N}_\mu }^2+C_\varepsilon e^{-2(S_0-c_0\mu^{3/2}-\varepsilon)/h}.
\end{aligned}
\ee
Now, for $h$ small enough, on ${\mathcal N}_\mu$ we have $|V_2-\Re\rho-\alpha^2(\nabla V_2)^2|\leq 2\mu_0 + \alpha^2|\nabla V_2|^2
$, and thus,
\be
\label{estpot}
\Vert ({\mathbf V}-\Re\rho -\alpha^2(\nabla V_2)^2)v\Vert_{{\mathcal N}_\mu }^2\leq (2\mu_0+c\alpha^2)^2\Vert v_2\Vert_{{\mathcal N}_\mu}^2+C''\Vert v_1\Vert_{{\mathcal N}_\mu }^2
\ee
As a consequence, by first choosing $C$ sufficiently large, then $\mu_0$ and $\alpha$ sufficiently small,  we can make true the inequality
\be
\label{choixconst}
\frac2{C}+2C(2\mu_0+c\alpha^2)^2\leq \frac1{4C_0},
\ee
where $C_0$ is the constant appearing in \eqref{estgrad}.

Inserting \eqref{estpot}-\eqref{choixconst} into \eqref{estQ2}, and using \eqref{estA}, for $h$ small enough we obtain (possibly by increasing $C''$),
\be
\label{estQ2bis}
h^2 |\la Q_2 v,v\ra_{{\mathcal N}_\mu } | \leq  C''\Vert A v\Vert_{{\mathcal N}_\mu }^2+\frac1{2C_0}\Vert  v_2\Vert_{{\mathcal N}_\mu }^2+C''\Vert  v_1\Vert_{{\mathcal N}_\mu }^2+C_\varepsilon e^{-2(S_0-c_0\mu^{3/2}-\varepsilon)/h}.
\ee
Since $Q_1$ is a vector-field, by similar (but rougher) arguments, we also have,
\be
\begin{aligned}
\label{estQ_1}
h |\la Q_1 v,v\ra_{{\mathcal N}_\mu } | & =\O(\Vert h\nabla v\Vert_{{\mathcal N}_\mu } \Vert  v\Vert_{{\mathcal N}_\mu })=\O(\Vert  h\nabla v\Vert_{{\mathcal N}_\mu }^2 + \Vert v\Vert_{{\mathcal N}_\mu }^2)\\
& = \O(\Vert  A v\Vert_{{\mathcal N}_\mu }^2 + \Vert v\Vert_{{\mathcal N}_\mu }^2+ e^{-2(S_0-c_0\mu^{3/2}-\varepsilon)/h})
\end{aligned}
\ee
and, of course,
\be
\label{estQ_0}
 |\la Q_0 v,v\ra_{{\mathcal N}_\mu } |=\O(\Vert  v\Vert_{{\mathcal N}_\mu }^2 ).
\ee
Using \eqref{estQ2bis}-\eqref{estQ_0} together with \eqref{A0B0} and \eqref{estgrad}, we obtain,
\be
\label{commA0B0}
\begin{aligned}
\Re\frac{i}2\la [A_0,B_0]v,v\ra_{{\mathcal N}_\mu } \geq \frac{\alpha h}{3C_0}\Vert  v_2\Vert_{{\mathcal N}_\mu }^2-C_1\alpha h&(\Vert  Av\Vert_{{\mathcal N}_\mu }^2+\Vert  v_1\Vert_{{\mathcal N}_\mu }^2)\\
& -C_\varepsilon e^{-2(S_0-c_0\mu^{3/2}-\varepsilon)/h},
\end{aligned}
\ee
with $C_1\geq 1$ independent of $h$, $\alpha$ and $\mu$, all of them small enough.

In addition, we also have,
$$
[{\mathbf W},B]=[{\mathbf W},B_0]-i\alpha h [{\mathbf W},{\mathbf R}],
$$
and thus, since $B_0$ is a scalar first-order semiclassical differential operator, and $R$ is 0-th order,
$$
[{\mathbf W},B]=\alpha h F_1(x,hD_x)
$$
where $F_1(x,hD_x)$ is a matrix of first-order semiclassical differential operators with bounded coefficients. As a consequence, by the same arguments as before, we have,
\be
\la h[{\mathbf W},B]v,v\ra =\O( h^2\Vert  Av\Vert_{{\mathcal N}_\mu }^2 +h^2 \Vert v\Vert_{{\mathcal N}_\mu }^2 + e^{-2(S_0-c_0\mu^{3/2}-\varepsilon)/h}).
\ee
Finally, since since $A_0$ is a scalar second-order semiclassical differential operator, we see that $h^{-1}[A_0,{\mathbf R}]$ is a matrix of first-order semiclassical differential operators with bounded coefficients, too, and thus,
\be
\label{commA0R}
\la i\alpha h[A_0,{\mathbf R}]v,v\ra=\O( h^2\Vert  Av\Vert_{{\mathcal N}_\mu }^2 +h^2 \Vert v\Vert_{{\mathcal N}_\mu }^2 + e^{-2(S_0-c_0\mu^{3/2}-\varepsilon)/h}).
\ee
Inserting \eqref{commA0B0}-\eqref{commA0R} into \eqref{formAB}, and using \eqref{imAB}, the result follows.
\end{proof}

For $h$ small enough, Lemma \ref{carleman} and \eqref{A+iB} imply,
\be
\label{Av2}
 \Vert Av\Vert^2_{{\mathcal N}_\mu} +\Vert Bv\Vert^2_ {{\mathcal N}_\mu} +h\Vert v_2\Vert^2_ {{\mathcal N}_\mu} =\O(h\Vert v_1\Vert^2_ {{\mathcal N}_\mu} + e^{-2(S_0-c_0\mu^{3/2}-\varepsilon)/h}).
 \ee
 Now, we prove,
 \begin{lemma}\sl 
 \label{LemmaAv}
 For any $C\geq 1$, there exists $C_1=C_1(C)>0$ such that, for all $h>0$ small enough, one has,
 $$
   \Vert Av\Vert_{{\mathcal N}_\mu}^2 \geq \frac{h}{C_1}  \Vert h\nabla v\Vert_{{\mathcal N}_\mu}^2+ \frac1{C_1}   \Vert v_1\Vert_{{\mathcal N}_\mu}^2  -\frac{h}C\Vert v_2\Vert_{{\mathcal N}_\mu}^2-C_1e^{-2(S_0-c_0\mu^{3/2}-\varepsilon)/h}.
 $$
 \end{lemma}
 \begin{proof}
 We have, 
  \be
 \begin{aligned}
  \Vert Av\Vert_{{\mathcal N}_\mu}^2=\Vert A_0v\Vert_{{\mathcal N}_\mu}^2+ \Vert h{\mathbf W}v\Vert_{{\mathcal N}_\mu}^2+2h\Im\la [A_0, {\mathbf W}]v, v\ra_{{\mathcal N}_\mu},
 \end {aligned}
 \ee
and thus, setting $A_j:= -h^2\Delta + V_j -\Re\rho -\alpha^2(\nabla V_2)^2$ ($j=1,2$),
 \be
 \label{Av}
 \begin{aligned}
  \Vert Av\Vert_{{\mathcal N}_\mu}^2 \geq  \Vert A_1v_1\Vert_{{\mathcal N}_\mu}^2+\Vert A_2v_2\Vert_{{\mathcal N}_\mu}^2+2h\Im\la [A_0, {\mathbf W}]v, v\ra_{{\mathcal N}_\mu}.
 \end {aligned}
 \ee
Moreover, by an integration by parts, 
 $$
  \begin{aligned}
 \la A_1v_1, v_1\ra_{{\mathcal N}_\mu} =\Vert h\nabla v_1\Vert_{{\mathcal N}_\mu}^2+\la (V_1 -\Re\rho -\alpha^2&(\nabla V_2)^2) v_1, v_1\ra_{{\mathcal N}_\mu}\\
 & +\O (e^{-2(S_0-c_0\mu^{3/2}-\varepsilon)/h}),
  \end {aligned}
 $$
and,  by Assumption 1, one has $\min V_1\left|_{\partial{\mathcal M}}\right. >0$. Therefore, by Cauchy-Schwarz inequality, for $\alpha$ and $h$ sufficiently small, we obtain the existence of a constant $C_1>0$ such that,
$$
\Vert A_1v_1\Vert_{{\mathcal N}_\mu}\cdot \Vert v_1\Vert_{{\mathcal N}_\mu}\geq \Vert h\nabla v_1\Vert_{{\mathcal N}_\mu}^2+\frac1{C_1}\Vert v_1\Vert_{{\mathcal N}_\mu}^2 +\O (e^{-2(S_0-c_0\mu^{3/2}-\varepsilon)/h}).
$$
Since also $\Vert A_1v_1\Vert_{{\mathcal N}_\mu}\cdot \Vert v_1\Vert_{{\mathcal N}_\mu}\leq 2C_1\Vert A_1v_1\Vert_{{\mathcal N}_\mu}^2+\frac1{2C_1}\Vert v_1\Vert_{{\mathcal N}_\mu}^2$,
we deduce,
\be
\label{A1vfin}
\Vert A_1v_1\Vert_{{\mathcal N}_\mu}^2\geq \frac1{2C_1}\Vert h\nabla v_1\Vert_{{\mathcal N}_\mu}^2+\frac1{4C_1^2}\Vert v_1\Vert_{{\mathcal N}_\mu}^2 +\O (e^{-2(S_0-c_0\mu^{3/2}-\varepsilon)/h}).
\ee
On the other hand, setting $\widetilde V_2:= V_2 -\Re\rho -\alpha^2(\nabla V_2)^2$, we have,
\be
\label{A2}
 \begin{aligned}
\Vert A_2v_2\Vert_{{\mathcal N}_\mu}^2 &=\Vert h^2\Delta v_2\Vert_{{\mathcal N}_\mu}^2 + \Vert \widetilde V_2v_2\Vert_{{\mathcal N}_\mu}^2+2\Im \la [-h^2\Delta, \widetilde V_2]v_2,v_2\ra_{{\mathcal N}_\mu}\\
&= \Vert h^2\Delta v_2\Vert_{{\mathcal N}_\mu}^2 + \Vert \widetilde V_2v_2\Vert_{{\mathcal N}_\mu}^2+\O(h\Vert h\nabla v_2\Vert_{{\mathcal N}_\mu} \Vert v_2\Vert_{{\mathcal N}_\mu}+h^2\Vert v_2\Vert_{{\mathcal N}_\mu}^2),
 \end{aligned}
\ee
and, still with an integration by parts and Cauchy-Schwarz inequality,
$$
\Vert h^2\Delta v_2\Vert_{{\mathcal N}_\mu}\Vert  v_2\Vert_{{\mathcal N}_\mu}\geq \la -h^2\Delta v_2,v_2\ra_{{\mathcal N}_\mu}=\Vert  h\nabla v_2\Vert_{{\mathcal N}_\mu}^2+\O (e^{-2(S_0-c_0\mu^{3/2}-\varepsilon)/h}).
$$
Then, writing,
$$
\Vert h^2\Delta v_2\Vert_{{\mathcal N}_\mu}\Vert  v_2\Vert_{{\mathcal N}_\mu}\leq \frac1{Ch}\Vert h^2\Delta v_2\Vert_{{\mathcal N}_\mu}^2 + Ch\Vert  v_2\Vert_{{\mathcal N}_\mu}^2,
$$
where $C\geq 1$ is arbitrary, we deduce,
\be
\label{h2Delta}
\Vert h^2\Delta v_2\Vert_{{\mathcal N}_\mu}^2\geq Ch \Vert h\nabla v_2\Vert_{{\mathcal N}_\mu}^2 -C^2h^2\Vert  v_2\Vert_{{\mathcal N}_\mu}^2+\O (e^{-2(S_0-c_0\mu^{3/2}-\varepsilon)/h}).
\ee
Inserting \eqref{h2Delta} into \eqref{A2}, we obtain the existence of a constant $C_2>0$ such that, for any $C\geq C_2$,
\be
 \begin{aligned}
\Vert A_2v_2\Vert_{{\mathcal N}_\mu}^2 \geq Ch \Vert h\nabla v_2\Vert_{{\mathcal N}_\mu}^2 -2C^2h^2\Vert  v_2\Vert_{{\mathcal N}_\mu}^2- & C_2h\Vert h\nabla v_2\Vert_{{\mathcal N}_\mu} \Vert v_2\Vert_{{\mathcal N}_\mu}\\
& +\O (e^{-2(S_0-c_0\mu^{3/2}-\varepsilon)/h}).
 \end{aligned}
\ee
In particular,
$$
 \begin{aligned}
\Vert A_2v_2\Vert_{{\mathcal N}_\mu}^2 \geq (C-C_2\sqrt{C})h \Vert h\nabla v_2\Vert_{{\mathcal N}_\mu}^2 - & (2C^2h^2 + \frac{C_2h}{\sqrt{C}})\Vert  v_2\Vert_{{\mathcal N}_\mu}^2\\
& +\O (e^{-2(S_0-c_0\mu^{3/2}-\varepsilon)/h}),
 \end{aligned}
$$
and thus, for $C$ sufficiently large,
\be
\label{A2vfin}
 \begin{aligned}
\Vert A_2v_2\Vert_{{\mathcal N}_\mu}^2 \geq \sqrt{C}h \Vert h\nabla v_2\Vert_{{\mathcal N}_\mu}^2 - & (2C^2h^2 + \frac{C_2h}{\sqrt{C}})\Vert  v_2\Vert_{{\mathcal N}_\mu}^2\\
& +\O (e^{-2(S_0-c_0\mu^{3/2}-\varepsilon)/h}).
 \end{aligned}
\ee
Finally, concerning the term $2h\Im\la [A_0, {\mathbf W}]v, v\ra_{{\mathcal N}_\mu}$ appearing in \eqref{Av}, since $A_0$ is scalar, we have,
$$
h\la [A_0, {\mathbf W}]v, v\ra_{{\mathcal N}_\mu}=\O(h^2)(\sum_{|\alpha|\leq 2}|\la  (hD_x)^\alpha v, v\ra_{{\mathcal N}_\mu}|),
$$
and thus, with an integration by parts, and still using \eqref{estbordN1v}-\eqref{estbordN2v},
\be
\label{A0W}
h\la [A_0, {\mathbf W}]v, v\ra_{{\mathcal N}_\mu}=\O(h^2)(\Vert h\nabla v\Vert_{{\mathcal N}_\mu}^2 + \Vert  v\Vert_{{\mathcal N}_\mu}^2)+\O (e^{-2(S_0-c_0\mu^{3/2}-\varepsilon)/h}).
\ee
Inserting \eqref{A1vfin}, \eqref{A2vfin}, and \eqref{A0W} into \eqref{Av}, we obtain the existence of a constant $C_3>0$ such that, for any $C\geq 1$ sufficiently large, and for $h$ sufficiently small (depending on $C$), one has,
$$
\begin{aligned}
 \Vert Av\Vert_{{\mathcal N}_\mu}^2 \geq \frac1{C_3}&\Vert h\nabla v_1\Vert_{{\mathcal N}_\mu}^2+ \frac1{C_3}\Vert  v_1\Vert_{{\mathcal N}_\mu}^2+Ch\Vert h\nabla v_2\Vert_{{\mathcal N}_\mu}^2- \frac{h}{C}\Vert v_2\Vert_{{\mathcal N}_\mu}^2\\
 &-C_3h^2\Vert h\nabla v\Vert_{{\mathcal N}_\mu}^2-C_3h^2\Vert  v\Vert_{{\mathcal N}_\mu}^2-C_3 e^{-2(S_0-c_0\mu^{3/2}-\varepsilon)/h}.
 \end{aligned}
$$
In particular, Lemma \ref{LemmaAv} follows.
\end{proof}

Gathering Lemma \ref{LemmaAv} and \eqref{Av2}, we obtain,
$$
 \frac{h}{C_1}  \Vert h\nabla v\Vert_{{\mathcal N}_\mu}^2+ \frac1{C_1}   \Vert v_1\Vert_{{\mathcal N}_\mu}^2  +h(1-\frac{1}C)\Vert v_2\Vert_{{\mathcal N}_\mu}^2 =\O(h\Vert v_1\Vert^2_ {{\mathcal N}_\mu} + e^{-2(S_0-c_0\mu^{3/2}-\varepsilon)/h}),
$$
where $C\geq 1$ is arbitrarily (and sufficiently) large, and $\varepsilon >0$ is arbitrary. Choosing $C\geq 2$, we conclude that, for any $\varepsilon >0$, for any $\mu >0$ sufficiently small, there exists $h_0=h_0(\varepsilon, \mu)>0$ such that, for $h\in (0,h_0]$, one has,
\be
\Vert h\nabla v\Vert_{{\mathcal N}_\mu}^2+  \Vert v\Vert_{{\mathcal N}_\mu}^2  =\O( e^{-2(S_0-c_0\mu^{3/2}-\varepsilon)/h}).
\ee
In particular,
\be
\label{estnmu/2}
\Vert h\nabla v\Vert_{{\mathcal N}_\frac{\mu}2}^2+  \Vert v\Vert_ {{\mathcal N}_\frac{\mu}2} ^2  =\O( e^{-2(S_0-c_0\mu^{3/2}-\varepsilon)/h}).
\ee
Now, recall from \eqref{defv} that, on ${\mathcal N}_\frac{\mu}2$, on has $|v|\geq e^{\frac12 \alpha\mu /h}|u|$. Therefore, \eqref{estnmu/2} implies,
$$
\Vert h\nabla u\Vert_{{\mathcal N}_\frac{\mu}2}^2+  \Vert u\Vert_ {{\mathcal N}_\frac{\mu}2} ^2  =\O( e^{-(2S_0+\alpha\mu-2c_0\mu^{3/2}-2\varepsilon)/h}).
$$
Here, $\alpha >0$ has been previously fixed sufficiently small, while $\mu$ can still be taken arbitrarily small. Therefore, by choosing $\mu>0$ in such a way that $4c_0\mu^{3/2} \leq \alpha\mu $ (that is $\mu \leq \alpha^2/(16c_0^2)$), and by taking $\varepsilon \leq c_0\mu^{3/2}$, we have proved,
\begin{proposition}\sl 
\label{acrossdM}There exists  $\delta_2>0$ and a neighborhood ${\mathcal N}$ of $\partial{\mathcal M}$, such that,
$$
\Vert u\Vert_{L^2({\mathcal N})}+\Vert h\nabla u\Vert_{L^2({\mathcal N})}=\O (e^{-(S_0+\delta_2)/h}).
$$
\end{proposition}

\section{Propagation across the crest $\partial\Omega_0$}

\subsection{Preliminaries}
In this section, we work near the minimal $d$-geodesic $\gamma\in G_1$ given by Assumption 5, and we will travel it, starting from $\partial{\mathcal M}$ up to $\partial U$. We denote by  $x_0\in \partial{\mathcal M}$ its starting point, by $y_0 \in \partial U$ its ending point, and by $x^{(1)}, \dots, x^{(N)}$ the sequence of points that constitute $\gamma\cap\partial\Omega_0$, ordered from the closest to $\partial{\mathcal M}$ up to the closest to $U$. We also denote by $\gamma^{(1)}, \gamma^{(2)}, \dots, \gamma^{(N +1)}$ the portions of $\gamma\backslash \left( \partial\Omega_0 \cup \partial{\mathcal M}\cup \partial U\right)$ that are in-between  $\partial{\mathcal M}$ and $x^{(1)}$, $x^{(1)}$ and $x^{(2)}$, ..., $x^{(N)}$ and $U$, respectively, in such a way that we have,
$$
\gamma = \{ x_0\} \cup\gamma^{(1)} \cup \{x^{(1)}\} \cup \gamma^{(2)}\cup \dots \cup \{x^{(N)}\}\cup \gamma^{(N +1)}\cup \{y_0\},
$$
where the unions are all disjoints.

Since $\gamma$ is a minimal $d$-geodesic,  we know that the function $\varphi$ is analytic in a neighborhood of $\gamma^{(1)}$ (that of course may shrink as we get closer to the end points of $\gamma^{(1)}$), and satisfies $|\nabla\varphi|^2 = V_2$ there. Moreover, by standard arguments of Riemannian geometry, $\gamma^{(1)}$ can be parametrized on some interval $(0, t_1)$ in such a way that, for all $t\in (0,t_1)$, one has,
\be
\label{parammgamma1}
\dot \gamma^{(1)} (t) = -\nabla\varphi \,(\gamma^{(1)}(t)).
\ee
Now, we set,
$$
v:= e^{\varphi /h}u.
$$
By Proposition \ref{acrossdM}, we know that there exists $\delta_2>0$ such that $u=\O(e^{-(S_0+\delta_2)/h})$ near $x_0$ (and analogous estimates for the derivatives of $u$). Hence (since $\varphi (x_0)=S_0$ and $\varphi \leq S_0$ everywhere), we obtain that $v$ is exponentially small near $\gamma^{(1)} (t)$ for all $t$ close enough to 0.

On the other hand, using the equation $Pu=\rho u$,  on a neighborhood of $\gamma^{(1)}$ we obtain,
$$
\left(-h^2\Delta + 2h(\nabla\varphi)\cdot \nabla -|\nabla\varphi|^2 +h(\Delta\varphi)+{\mathbf V}+h{\mathbf W}-\rho \right)v=0,
$$
that we prefer to write as,
$$
\left(\frac{i}2 h^2\Delta + h(\nabla\varphi)\cdot D_x +\frac{i}2 V_2 -\frac{ih}2 (\Delta\varphi)-\frac{i}2{\mathbf V}-\frac{ih}2{\mathbf W}+\frac{i\rho }2 \right)v=0,
$$
that is,
\be
\label{eqw}
Q(x, hD_x; h)v=0,
\ee
where $Q(x, hD_x; h)$ is a semiclassical differential matrix-operator, with diagonal principal symbol $q_0={\rm diag}(q_1,q_2)$ given by,
$$
\begin{aligned}
& q_1(x,\xi)=(\nabla\varphi)\cdot\xi-\frac{i}2\xi^2+\frac{i}2(V_2-V_1)\\
& q_2(x,\xi)=(\nabla\varphi)\cdot\xi-\frac{i}2\xi^2.
\end{aligned}
$$
Since $V_1>V_2$ near $\gamma^{(1)}$, we see that $q_1$ is elliptic there, while, thanks to \eqref{parammgamma1}, the curve $t\mapsto (\gamma^{(1)}(t),0)$ is a (real) integral curve of the Hamilton flow $H_{q_2}=(\nabla_\xi q_2, -\nabla_x q_2)$ of $q_2$. As a consequence, Equation \eqref{eqw} can actually be locally reduced to a scalar one (e.g. by using a parametrix of $Q_1$), and we can apply standard results on the micro-support (see, e.g., \cite[Chapter 4]{Ma2} or \cite[Section 5]{Ma1}). In particular, by Hanges' Theorem of propagation, and since $v$ is exponentially small near $\gamma^{(1)} (t)$ for $t$ small, we deduce (denoting by $MS(v)$ the micro-support of $v$),
\be
\label{propag}
\gamma^{(1)}\times \{0\} \, \cap \, MS(v) =\emptyset.
\ee
But since $q_2^{-1}(0) \subset \{ \xi =0\}$, we also know that $MS(v)\subset \{ \xi =0\}$, and thus, \eqref{propag} implies, 
$$
\gamma^{(1)}\times \R^n \, \cap \, MS(v) =\emptyset.
$$
Still by standard results of analytic microlocal analysis (see, e.g. \cite{Sj, Ma2}), we conclude that $v$ is exponentially small on a whole neighborhood of 
$\gamma^{(1)}$ as $h\to 0_+$, that is, we have proved,
\begin{proposition}\sl
\label{estgamma1}
For any $x\in \gamma^{(1)}$, there exists $\delta =\delta(x)>0$ such that $u e^{\varphi /h}=\O(e^{-\delta /h})$ on a neighborhood of $x$, uniformly for $h>0$ sufficiently small.
\end{proposition}

The next step will consist in crossing $\partial \Omega_0$ at the ending point $x^{(1)}$ of $\gamma^{(1)}$.

By  \cite{Pe}, Theorem 2.14, we know that there exists a neighborhood ${\mathcal V}_1$ of $x^{(1)}$ and two positive real-analytic functions $\varphi_1, \varphi_2$ on ${\mathcal V}_1$, such that,

\be
\label{phi12}
\begin{aligned}
&\varphi_1 = \varphi \mbox{ on } {\mathcal V}_1\cap\{ V_1 <V_2\};\\
&\varphi_2= \varphi \mbox{ on } {\mathcal V}_1\cap\{ V_2 <V_1\};\\
&|\nabla\varphi_j(x)|^2 = V_j(x) \quad (j=1,2);\\
&\varphi_1 =\varphi_2 \mbox{ and } \nabla\varphi_1 =\nabla\varphi_2 \mbox{ on } {\mathcal V}_1\cap\partial\Omega_0;\\
&\varphi_1(x) -\varphi_2(x) \sim d(x, \partial\Omega_0)^2.
\end{aligned}
\ee

Actually,  $\varphi_1$ is nothing but the analytic continuation of $\varphi(x)$ from ${\mathcal V}_1\cap \{V_1 <V_2\}$ obtained by using the  metric $V_1(x)dx^2$ near $x^{(1)}$ (instead of $\min (V_1,V_2)dx^2$), while $\varphi_2$ is the (analytic) phase function of the Lagrangian manifold obtained as the flow-out of $\{(x,\nabla\varphi_1(x))\, ;\, x\in {\mathcal V}_1\cap\partial\Omega_0 \}$ under the Hamilton flow of $q_2(x,\xi):=\xi^2 - V_2(x)$. 

Now, we set,
$$
v:= e^{\varphi_2/h} u,
$$
and
$$
{\mathcal V}_1^\pm := {\mathcal V}_1 \cap \{ \pm V_1 > \pm V_2\},
$$
By \eqref{agmon} and standard Sobolev estimates, we know that, for all $\varepsilon >0$, $w$ is $\O(e^{(\varphi_2 -\varphi +\varepsilon) /h})$ together with all its derivatives. In particular, since, by \eqref{phi12},  $\varphi = \varphi_1 >\varphi_2$ on ${\mathcal V}_1^-$,  we obtain that $v$ is locally exponentially small there. Moreover, by  Proposition \ref{estgamma1}, $w$ is also locally exponentially small in ${\mathcal V}_1^+$ (on which $\varphi =\varphi_2$). Thus,  the microsupport of $v$ satisifies,
\be
\label{msw1}
MS(v)\cap \{ x\in {\mathcal V}_1\} \subset \{ V_1=V_2\}.
\ee

On the other hand, $v$ is solution to $Qv=\rho v$ with,
$$ 
Q:= e^{\varphi_2/h}P e^{-\varphi_2/h},
$$
and the semiclassical principal symbol of $Q$ is given by,
$$
\begin{aligned}
q_0(x,\xi) & = \left(\begin{array}{cc}
(\xi +i\nabla\varphi_2)^2 + V_1 & 0\\
0 & (\xi +i\nabla\varphi_2)^2 + V_2 
\end{array}
\right)\\
& =\left(\begin{array}{cc}
\xi^2 + V_1-V_2 +2i\xi\cdot\nabla\varphi_2 & 0\\
0 & \xi^2+2i\xi\cdot\nabla\varphi_2  
\end{array}
\right).
\end{aligned}
$$
In particular, the characteristic set ${\rm Char}(Q):=(\det q_0)^{-1}(0)$ of $Q$ satisfies,
\be
\label{charQ}
{\rm Char}(Q) \subset \{ \xi =0\} \cup \{\xi^2 =V_2(x)-V_1(x)\},
\ee
and since we know that $MS(v)\subset {\rm Char}(Q)$ (see, e.g., \cite{Ma2}), we conclude from \eqref{msw1}-\eqref{charQ} that we have,
\be
\label{msw}
MS(v)\cap \{ x\in {\mathcal V}_1\} \,\,  \subset\, \, \{V_1(x)=V_2(x)\, , \, \xi =0\}=\partial\Omega_0 \times\{0\}.
\ee

Therefore, in order to prove that $v$ is exponentially small near $x^{(1)}$, it remains to show,
\begin{proposition}\sl One has,
\label{caracdouble}
$$
(x^{(1)},0)\notin MS(v).
$$
\end{proposition}
\begin{remark}\sl This result can be seen as a propagation of the microsupport for a (matrix-) operator with non-involutive double characteristics, and has connections with the results proved in \cite{LL, PeSo}.
\end{remark}

\subsection{Proof of Proposition \ref{caracdouble}}
To begin with, let us write in a more precise way the expression of $Q$. We have,
$$
\begin{aligned}
Q & = \left(\begin{array}{cc}
(hD_x +i\nabla\varphi_2)^2 + V_1 & hr_0(x) + ihr_1(x)\cdot (hD_x+i\nabla\varphi_2)\\
h{\overline r_0}(x) - ih (hD_x+i\nabla\varphi_2)\cdot {\overline r_1}(x) & (hD_x +i\nabla\varphi_2)^2 + V_2
\end{array}
\right),
\end{aligned}
$$
and thus, the equation $Qv=\rho v$ can be written as,
\be
\label{systcarct2}
\left\{
\begin{aligned}
&(Q_1-\rho)v_1 = hA_1v_2\, ;\\
&(Q_2-\rho)v_2 =hA_2 v_1,
\end{aligned}
\right.
\ee
with,
$$
\begin{aligned}
& Q_1:=-h^2\Delta +V_1-V_2 +2i \nabla\varphi_2\cdot hD_x +h\Delta\varphi_2 \, ; \\
& Q_2:=-h^2\Delta  +2i \nabla\varphi_2\cdot hD_x +h\Delta\varphi_2  ,
\end{aligned}
$$
and,
$$
\begin{aligned}
& A_1:=-r_0(x) + ir_1(x)\cdot (hD_x+i\nabla\varphi_2)\, ;\\
& A_2:= -{\overline r_0}(x) + i {\overline r_1}(x)\cdot (hD_x+i\nabla\varphi_2) +h\,{\rm div}\, {\overline r_1}(x).
\end{aligned}
$$
(In particular, by Assumption 6, $A_2$ is elliptic at $(x^{(1)}, 0)$.)

We plan to transform $Q_2-\rho$ into $ihD_{z_n}$ microlocally near $(x^{(1)}, 0)$ by the use of a convenient Fourier-Bros-Iagolnitzer transform in the same spirit as in \cite[Section 7]{Sj}. For this purpose, we first take local (real-analytic) coordinates $x=(x',x_n)\in \R^{n-1}\times \R$ centered at $x^{(1)}$ in such a way that the hypersurface $\partial\Omega_0 =\{ V_1=V_2\}$ becomes $\{ x_n=0\}$ (and $\{V_1 >V_2\}$ becomes $\{ x_n>0\}$). Then, using Assumption 5, we see that $\nabla\varphi_2$ is transversal to $\{ V_1=V_2\}$, and thus (since it also points-out towards $\{V_1 >V_2\}$) another change of analytic coordinates transforms the vector-field $2 \nabla\varphi_2(x)\cdot \nabla_x$ into $\partial_{x_n}$. Therefore, in the new local coordinates (that we still denote by $(x',x_n)$), the operators $Q_1$ and $Q_2$ become,
$$
\begin{aligned}
& Q_1=\Lambda(x,hD_x) +ihD_{x_n} +x_ng(x)  \, ; \\
& Q_2=\Lambda(x,hD_x) +ihD_{x_n},
\end{aligned}
$$
where  $\Lambda(x,hD_x)$ is a second-order semiclassical differential operator with non-negative real-valued principal symbol $\Lambda_0(x,\xi)$ homogeneous with respect to $\xi$, and where the function $g$ is real-analytic and positive at $x=0$.

We first get rid of $\rho$ in the system \eqref{systcarct2}. For this purpose, we consider the function $\sigma =\sigma (x)$ solution to,
\be
\label{elimrho}
\left\{
\begin{aligned}
& i\frac{\partial \sigma}{\partial x_n} +\Lambda_0(x,\nabla_x\sigma  )=\rho\, ;\\
& \sigma \left\vert_{x_n=0} =0,
\right.
\end{aligned}\right.
\ee
and we set,
$$
w := e^{-i\sigma /h} v.
$$
Since $\sigma =\O(\rho)$ and $\rho=\rho(h)\to 0$ as $h\to 0$, for any $\varepsilon >0$ we  still have $w =\O(e^{\varepsilon /h})$, and $w$ is solution   to the system,
\be
\label{systcarct3}
\left\{
\begin{aligned}
&\widetilde Q_1w_1 = h\widetilde A_1w_2\, ;\\
&\widetilde Q_2w_2 =h\widetilde A_2 w_1,
\end{aligned}
\right.
\ee
with,
$$
\begin{aligned}
& \widetilde Q_1:=\widetilde \Lambda(x,hD_x) +ihD_{x_n} +x_ng(x) \, ; \\
& \widetilde Q_2:=\widetilde  \Lambda(x,hD_x) +ihD_{x_n},
\end{aligned}
$$
where $\widetilde \Lambda(x,hD_x)$ is a second-order semiclassical differential operator, with principal symbol,
\be
\label{Lambda0tilde}
\widetilde \Lambda_0 (x,\xi) =\Lambda_0 (x,\xi)+(\nabla_\xi \Lambda_0) (x,\xi)\cdot \nabla_x\sigma (x),
\ee
and with $\widetilde A_2$ elliptic at $(0,0)$.

In particular, we have,
\be
\label{estLambdatilde}
\widetilde \Lambda_0 (x,\xi)=\O(|\xi|^2 + |\rho\xi|))\quad ;\quad \nabla_\xi \widetilde \Lambda_0 (x,\xi)=\O(|\xi|+|\rho|).
\ee
At this point, it may be useful to make a few considerations, in order to explain that the reasons for which the computations below (that may seem somehow obscure) do work, are actually rather natural. A way to understand this consists in considering the  toy-model system,
$$
\left\{
\begin{aligned}
&(ihD_{x_n} + x_n)w_1 = hw_2\, ;\\
&ihD_{x_n}w_2 =h w_1.
\end{aligned}
\right.
$$
Substituting $w_1=iD_{x_n}w_2$ in the first equation, we obtain $Qw_2=0$ where the Weyl-symbol $q$ of $Q$ satisfies,
$$
q(x,\xi;h) = -\xi_n^2 +ix_n\xi_n -\frac{h}2 + \O(h^2).
$$
If in addition we conjugate $Q$ with $e^{-2\delta x_n^2/h}$ with $\delta >0$ small enough, then we obtain a differential operator with principal symbol $\widetilde q$ satisfying,
\be
\label{toy}
\Re \widetilde q(x,\xi;h) \leq -\xi_n^2- \delta x_n^2  -\frac{h}4.
\ee
Therefore, if the solution is exponentially small for $x_n\not =0$, it becomes natural to expect it to be (microlocally) exponentially small at $(0,0)$, too.

In several dimensions and in the analytic setting, however, the technical procedure is more complicated, and in order to eliminate the term $\widetilde \Lambda (x,hD_x)$, it will be necessary to use a Fourier integral operator with complex phase-function.

The microlocal transformation of $\widetilde Q_2$ will also take into account the operator $\widetilde Q_1$, in the sense that we will do it in such a way that the operator of multiplication by $x_ng(x)$  keeps a nice form.
In order to do so,  we consider the holomorphic function $\psi_0=\psi_0(x,z')$, solution to the system,
\be
\label{eikonFBIzero}
\left\{
\begin{aligned}
& \frac{\partial \psi_0}{\partial x_n} +i\widetilde  \Lambda_0(x,\nabla_x\psi_0 )=i\mu x_n\, ;\\
& \psi_0 \left\vert_{x_n=0} = i\mu (z'-x')^2/2.
\right.
\end{aligned}\right.
\ee
(Here, $\mu>0$ is a constant that will be fixed small enough later on.)

Then, we consider the holomorphic function $\psi =\psi (x,z)$ defined on a complex neighborhood of $(0,0)$, solution to the eikonal system,
\be
\label{eikonFBI}
\left\{
\begin{aligned}
&\frac{\partial \psi}{\partial z_n} =-i\widetilde \Lambda_0(x,\nabla_x\psi )- \frac{\partial \psi}{\partial x_n} \, ;\\
& \psi \left\vert_{z_n=0} = \psi_0.
\right.
\end{aligned}\right.
\ee
We prove,
\begin{lemma}\sl
\label{lemmapsimu}
One has,
$$
\begin{aligned}
&\psi (z,x) = i\mu \frac{(z-x)^2}{2} + x_n\alpha (x',z')\\
& \hskip 3cm+\O\left( |(x_n,z_n)|^2(|x'-z'|+|\rho|) + |(x_n,z_n)|^3\right);\\
& \nabla_{z'}\psi (z,x) =i\mu (z'-x') + x_n\nabla_{z'}\alpha (x',z') +\O(|(x_n,z_n)|^2);\\
& \partial_{z_n}\psi (z,x) =i\mu (z_n-x_n)+\O(|z_n|(|x'-z'|+|\rho|) +|(x_n,z_n)|^2).
\end{aligned}
$$
uniformly in a neighborhood of $(0,0)$, where $\alpha=\alpha (x',z')$ is the unique complex solution near $0$ to the equation,
\be
\label{eqalpha}
\alpha + i\Lambda_0(x',0 ; i\mu (x'-z'), \alpha) =0.
\ee
\end{lemma}
\begin{proof} From \eqref{eikonFBIzero} and the definition of $\alpha$, we immediately obtain,
$$
\frac{\partial \psi_0}{\partial x_n}\left\vert_{x_n=0} \right.= \alpha,
$$
and, using \eqref{estLambdatilde}, we see that,
\be
\label{propalpha}
\begin{aligned}
& \alpha = \O(|x'-z'|^2+|\rho (x'-z')|)\, ; \\
&  |\nabla_{x'}\alpha |+|\nabla_{z'} \alpha| = \O(|x'-z'|+|\rho|).
\end{aligned}
\ee

Then, differentiating \eqref{eikonFBIzero} with respect to $x_n$, we deduce,
$$
\frac{\partial^2 \psi_0}{\partial x_n^2}\left\vert_{x_n=0} \right.=i\mu +\O(|x'-z'| +|\rho|).
$$
and thus, by Taylor's formula,
\be
\label{exprpsi0}
\begin{aligned}
& \psi_0(x,z') =i\mu\frac{(z'-x')^2}2 + x_n\alpha(x',z')+i\mu\frac{x_n^2}2+\O( (|x'-z'|+|\rho|)x_n^2+|x_n|^3 );\\
& \nabla_x \psi_0 (x,z') = \O( |x'-z'| + |x_n|).
\end{aligned}
\ee
On the orther hand, using  \eqref{eikonFBI}, we obtain,
$$
 \frac{\partial \psi}{\partial z_n}\left\vert_{z_n=0} \right.= -i\widetilde \Lambda_0(x,\nabla_x\psi_0)-\frac{\partial \psi_0}{\partial x_n}=-i\mu x_n.
$$
In particular, $\nabla_{z'} \partial_{z_n} \psi \left\vert_{z_n=0} \right.=0$, and, differentiating  \eqref{eikonFBI},
$$
\begin{aligned}
 \frac{\partial^2 \psi}{\partial z_n^2}\left\vert_{z_n=0} \right. & = -i(\nabla_\xi \widetilde \Lambda_0)(x,\nabla_x\psi_0)\cdot \nabla_x \frac{\partial\psi}{\partial z_n}\left\vert_{z_n=0} \right. -\partial_{x_n}\frac{\partial \psi_0}{\partial z_n}\left\vert_{z_n=0} \right. ;\\
& =i\mu+ \O(|\nabla_x\psi_0| + |\rho|))=i\mu +\O(|x'-z'|+|x_n|+|\rho|).
\end{aligned}
$$
The result follows by Taylor's formula.
\end{proof}

Now, by \cite[Theorem 9.3]{Sj}, we can find an analytic symbol $a(x,z;h)$ defined in a  complex neighborhood of $(0,0)$, such that,
\be
\label{constsymbTmu}
e^{-i\psi(x,z)/h}\left( ihD_{z_n} - {}^tQ_2(x,hD_x) \right)\left( (a(x,z;h)e^{i\psi(x,z)/h}\right)=\O(e^{-2\delta_0/h}),
\ee
where ${}^tQ_2(x,hD_x) $ stands for the formal transposed of $Q_2$.

Then, fixing a cut-off function $\chi_0=\chi_0(x)$ on a small enough neighborhood of $0$,  we set,
\be
T w (z;h):= \int_{\R^n} e^{i\psi (x,z)/h} a (x,z ;h) \chi_0(x)w(x) dx.
\ee
By Lemma \ref{lemmapsimu} and \eqref{propalpha}, we have,
$$
\nabla_x\psi (0,0) = 0\quad ; \quad \Im \nabla_x^2 \psi (0,0) >0\quad ; \quad \det \nabla_z\nabla_x\psi (0,0) \not= 0,
$$
and thus, by the general theory of \cite{Sj} (and its obvious generalization for the study of the microsupport near the nul section of $T^*\R^n$: see, e.g., \cite{Ma2}), $T$ is a Fourier-Bors-Iagolnitzer transform that can be used to characterize the microsupport of $h$-dependent functions near $(0,0)$: Setting,
$$
\Phi (z):= \sup_{x\in \R^n}(-\Im \psi (x,z)),
$$
and introducing the local (complex) canonical transformation,
$$
\kappa \, :\, (x, -\nabla_x\psi) \mapsto (z,\nabla_z\psi),
$$
then, for any $w=w(x;h)$ (say, in $L^\infty ({\mathcal V}_0)$ where ${\mathcal V}_0$ is a neighborhood of $0$ containing the support of $\chi_0$), such that $\Vert w\Vert_{L^\infty ({\mathcal V}_0)} =\O(e^{\varepsilon /h})$ for all $\varepsilon >0$ and $h$ small enough, and for $(x_0,\xi_0)\in \R^{2n}$ close enough to $(0,0)$, we have the equivalence,
\be
\label{caractMS}
(x_0,\xi_0)\notin MS(w)\, \Leftrightarrow \left\{
 \begin{array}{l}
 \mbox{There exists }\delta >0\, \mbox{ and a complex neighborhood }\\
  \omega_0\, \mbox{ of the $z$-projection of } \kappa (x_0,\xi_0), \mbox{ such that } \\
 \Vert e^{-\Phi /h}T w \Vert_{L^2(\omega_0)}=\O(e^{-\delta)/h}) \mbox{ uniformly for } \\ 
h \mbox{ small enough.} 
\end{array}\right.
\ee
Moreover, $Tw$ belongs to the Sj\"ostrand  space $H_{\Phi, 0}$ of $h$-dependent holomorphic functions $v=v(z;h)$ near $z=0$, such that, for all $\varepsilon >0$, $v=\O(e^{(\Phi +\varepsilon)/h})$ uniformly as $h\to 0_+$. 

Now, by construction (in particular \eqref{constsymbTmu}), near $z=0$ we have,
\be
\label{transfQ2}
T \widetilde Q_2w = ihD_{z_n}T w +\O(e^{(\Phi -\delta_0)/h}).
\ee
On the other hand, by \cite[Proposition 7.4]{Sj}, we also have,
\be
\label{transfxn}
T (x_ng(x)w) =L T w,
\ee
where $L=L(z, hD_z)$ is a pseudodifferential operator in the complex domain (in the sense of \cite{Sj}), with principal symbol $l_0(z,\zeta)$ such that,
\be
\label{b0k}
l_0(\kappa (x,\xi) )= x_ng(x).
\ee
We show,
\begin{lemma}\sl 
\label{approxl0}
One has,
$$
l_0(z,\zeta )= (z_n+ \frac{i}{\mu}  \zeta_n)g(z'+\frac{i}{\mu} \zeta')+ \O\left(|(z_n,\zeta_n)|(|\zeta'|+|\rho|)+|(z_n,\zeta_n)|^2 \right).
$$
\end{lemma}
\begin{proof}
By definition, if $(z,\zeta) =\kappa (x,\xi)$, then $\xi = -\nabla_x\psi (x,z)$ and $\zeta = \nabla_z\psi (x,z)$, and thus, by Lemma \ref{lemmapsimu},
$$
\begin{aligned}
& \zeta_n = i\mu (z_n-x_n)+\O\left(|z_n|(|x'-z'|+|\rho|)+|(x_n,z_n)|^2\right);\\
& \zeta' = i\mu (z'-x')+x_n\nabla_{z'}\alpha (x',z')+\O\left(|(x_n,z_n)|^2 \right).
\end{aligned}
$$
In view of \eqref{propalpha}, we first deduce,
$$
\begin{aligned}
& x_n = \O\left(|(z_n,\zeta_n)|\right) ;\\
& x' = \O\left(|(z',\zeta')|+|\rho|\cdot  |(z_n,\zeta_n)|+|(z_n,\zeta_n)|^2\right),
\end{aligned}
$$
then,
$$
\begin{aligned}
& x_n=z_n+\frac{i}{\mu} \zeta_n + \O\left(|z_n|(|x'-z'|+|\rho|)+|(z_n,\zeta_n)|^2 \right);\\
& x' =z'+\frac{i}{\mu} \zeta'+\O(|(z_n,\zeta_n)|(|x'-z'| +|\rho|)+|(z_n,\zeta_n)|^2),
\end{aligned}
$$
and finally,
$$
\begin{aligned}
& x_n=z_n+\frac{i}{\mu} \zeta_n + \O\left(|(z_n,\zeta_n)|(|\zeta'|+|\rho|)+|(z_n,\zeta_n)|^2 \right);\\
& x' =z'+\frac{i}{\mu} \zeta'+\O(|(z_n,\zeta_n)|(|\zeta'| +|\rho|)+|(z_n,\zeta_n)|^2).
\end{aligned}
$$
Inserting this into \eqref{b0k}, the result follows.
\end{proof}

Now, we set, 
$$
\widehat w= (\widehat w_1,\widehat w_2):= T w.
$$
By \eqref{transfQ2} and \eqref{transfxn},
the system \eqref{systcarct3} becomes
\be
\label{systcarct2transf}
\left\{
\begin{aligned}
&(ihD_{z_n}+ L)\widehat w_1 = h\widehat A_1\widehat w_2\, ;\\
&ihD_{z_n}\widehat w_2 =h\widehat A_2\widehat w_1,
\end{aligned}
\right.
\ee
where $\widehat A_1, \widehat A_2$ are 0-th order pseudodifferential operators in the complex domain, with $\widehat A_2$ elliptic at $(0,0)$.

We also need a general result on the pseudodiffential operators in the complex domain. Let $A(z, hD_z)$ be such a operator, acting on some Sj\"ostrand's space $H_{\phi, z_0}$ with $\phi \in C^2$. Then (see \cite[Chapter 4]{Sj}), $A$ can be written as,
$$
Au(z) =\frac1{(2\pi h)^n}\iint_{\Gamma (z)} e^{i(z-y)\zeta /h}a(z,\zeta ;h)u(y;h)dy\, d\zeta,
$$
where $a$ is an analytic symbol defined near $(z_0, \frac2{i}\nabla\phi (z_0))$, and $\Gamma (z)$ is the complex contour of $\C^{2n}$ given by,
\be
\label{Gammaphi}
\Gamma (z)\, :\, \left\{
\begin{aligned}
& \zeta = \frac2{i}\nabla\phi (z) +iR(\overline{z-y})\, ;\\
& y\in \C^n,\,\, |z-y| \leq r,
\end{aligned}
\right.
\ee
where $R>0$ is chosen sufficiently large with respect to the Hessian of $\phi$, and $r>0$ is chosen sufficiently small in such a way that $u$ is well defined in the ball centered at $z_0$ with radius $2Rr$ (in which case the previous formula defines $Au(z)$ for $z$ in the  ball centered at $z_0$ with radius $Rr$).

In all these formulas, $\nabla := \frac12(\nabla_{\Re z} -i \nabla_{\Im z})$ stands for the holomorphic gradient.

We have,
\begin{lemma}\sl
\label{approxpseudo}
$$
Au(z) = a(z,\frac2{i}\nabla\phi (z))u(z) +( \nabla_\zeta a)(z,\frac2{i}\nabla\phi (z))(hD_z - \frac2{i}\nabla\phi (z))u(z)+v(z),
$$
with,
$$
\Vert e^{-\phi /h}v\Vert _{L^2(|z-z_0|<Rr)} \leq Ch\Vert e^{-\phi /h}u\Vert _{L^2(|z-z_0|<2Rr)},
$$
where the positive constant $C$ can be taken independent of $R$ and $r$, as long as $Rr$ remains small enough.
\end{lemma}
\begin{proof}
For $\zeta = \frac2{i}\nabla\phi (z)+iR(\overline{x-y})$, we write,
$$
\begin{aligned}
a(z,\zeta)& =a(z,\frac2{i}\nabla\phi (z))+iR( \nabla_\zeta a)(z,\frac2{i}\nabla\phi (z))\cdot (\overline{x-y})+\O(|z-y|^2)\\
& = a(z,\frac2{i}\nabla\phi (z)) +( \nabla_\zeta a)(z,\frac2{i}\nabla\phi (z))\cdot (\zeta - \frac2{i}\nabla\phi (z))+\O(|z-y|^2),
\end{aligned}
$$
where the $\O(|z-y|^2)$ only depends on the behaviour of $a$ near $z_0$.
Then, we observe that, since $\Gamma (z)$ is a good contour (in the sense of \cite{Sj}), along $\Gamma (z)$ we have,
$$
\left| e^{\phi(z)/h}e^{i(z-x)\zeta /h}e^{-\phi(y)/h}\right| =\O(e^{-c|z-y|^2/h})
$$
with $c>0$ constant, and the result follows by an application of the Schur Lemma (see, e.g., \cite{Ma2}) and the fact that,
$$
h^{-n}\int_{x\in \C^n}|x|^2 e^{-c|x|^2/h} L(dx) =\O(h).
$$
(Here, $L(dx)$ stands for the Lebesgue measure on $\C^n$.)
\end{proof}

In order to perform exponentiel weighted estimates with $Tw$, we have to specify better the function $\Phi$. We show,
\begin{lemma}\sl
\label{exprPhi}
$$
\Phi (z) = \mu \frac{(\Im z)^2}2 +\O\left(   |\rho|^4+|\Im z'|^4+|z_n|( |\Im z'|^2+|\rho|^2)+|z_n|^3\right).
$$
\end{lemma}
\begin{proof}
By definition, $\Phi (z) =-\Im \psi (x_c(z), z)$, where $x_c(z)$ is the (real) critical point of $x\mapsto -\Im \psi (x,z)$. By computations similar to those of Lemma \ref{lemmapsimu}, we see,
$$
\begin{aligned}
& \nabla_{x'}\psi (x,z) = i\mu (x'-z') + x_n \nabla_{x'}\alpha (x',z') + \O(|(x_n,z_n)|^2);\\
& \partial_{x_n}\psi (x,z) =\alpha (x',z') +i\mu (x_n-z_n) + \O(|(x_n,z_n)|^2).
\end{aligned}
$$
Hence, using \eqref{propalpha}, the critical point $x_c(z)=(x_c', x_n^c)$ of $\Im \psi$ satisfies,
$$
\begin{aligned}
& x_c' -\Re z'=\O\left(|x_n^c|(|x_c' - z'| +|\rho|)+ |(x_n^c,z_n)|^2\right);\\
& x_n^c-\Re z_n = \O\left(|x'_c-z'|(|x_c' - z'| +|\rho|)+|(x_n^c,z_n)|^2\right).
\end{aligned}
$$
We first deduce,
$$
\begin{aligned}
& x_c' -\Re z'=\O\left(|x_n^c|(|\Im z'| +|\rho|)+ |(x_n^c,z_n)|^2\right);\\
& x_n^c = \O\left(|z_n|+|x'_c-z'|(|x_c' - z'| +|\rho|)\right).
\end{aligned}
$$
and then,
\be
\label{xc-Rez}
\begin{aligned}
& x_c' -\Re z'=\O\left(|z_n|(|\Im z'| +|\rho|)+|\Im z'|(|\Im z'|^2 +|\rho|^2)+ |z_n|^2\right);\\
& x_c' - z'=\O\left(|\Im z'| +|\rho z_n|+ |z_n|^2\right);\\
& x_n^c-\Re z_n = \O\left(|\Im z'|^2+|\rho \Im z'|+|\rho|^2|z_n| + |z_n|^2\right);\\
& x_n^c = \O\left(|z_n| +|\Im z'|^2+|\rho \Im z'|\right).
\end{aligned}
\ee
By Lemma \ref{lemmapsimu}, we also have,
$$
\begin{aligned}
\Phi (z) = \frac{\mu}2 (\Im z)^2 & + \O( |x_c-\Re z|^2 + |x_n^c\alpha (x'_c,z')| )\\
& + \O(  |(x_n^c,z_n)|^2(|x_c'-z'|+|\rho|) + |(x_n^c,z_n)|^3)
\end{aligned}
$$
and thus, using \eqref{xc-Rez} and \eqref{propalpha}, we find,
$$
\begin{aligned}
\Phi (z) =  \mu \frac{(\Im z)^2}2 & +\O\left( |\Im z'|^4 + |\rho|^2|\Im z'|^2+|\Im z'|^3|\rho| \right)\\
& +\O\left(|z_n|(|\Im z'|^2 +|\rho\Im z'|)+ |z_n|^2(|\Im z'| +|\rho|)+ |z_n|^3\right),
\end{aligned}
$$
and the result follows by Cauchy-Schwarz inequality.
\end{proof}
In particular, in a small enough neighborhood of $0$, we deduce from Lemma \ref{exprPhi},
$$
\Phi (z)\leq \mu \frac{(\Im z)^2}2+\frac12|\Im z'|^2 + C|z_n|^3+|\rho|^2,
$$
where $C>0$ is a constant. For $\delta >0$ small enough, we set,
\be
\label{defPhidelta}
\Phi_\delta (z):= \mu \frac{(\Im z_n)^2}2+2| z'|^2 + C|z_n|^3+|\rho|^2 -\delta (\Re z_n)^2.
\ee
In particular $\Phi_\delta$ is $C^2$, and if $\mu$ has been chosen $\leq 1$, we have,
\be
\label{uppbdPhi}
\Phi(z) \leq \Phi_\delta(z) - | z'|^2 +\delta (\Re z_n)^2.
\ee
Now, we fix $\delta_0>0$ sufficiently small, and we set,
\be
\label{defOmega}
\Omega := \{ z\in \C^n\, ;\, |z_n| < \delta_0,\, | z'| < \delta_0\},
\ee
and we observe,

- \underline{On $\{ |z_n| =\delta_0\}\cap \Omega$}: Analysing $\kappa^{-1} (z,\zeta)$, we see that these points correspond to $(x,\xi)$ such that $|(x_n,\xi_n)|\geq \delta_0/2$, and thus, by \eqref{msw}, there exists $\delta_1>0$ such that, for such points,
$$
\widehat w(z;h) =\O(e^{(\Phi (z) -2\delta_1)/h}).
$$
Therefore, by \eqref{uppbdPhi},
$$
\widehat w(z;h) =\O(e^{(\Phi_\delta (z) -2\delta_1+\delta\delta_0^2)/h}),
$$
and thus, if we choose $\delta \leq \delta_1\delta_0^{-2}$,
\be
\widehat w(z;h) =\O(e^{(\Phi_\delta (z) -\delta_1)/h}).
\ee

- \underline{On $\{ | z'| =\delta_0\}\cap \Omega$}: For any $\varepsilon >0$, we have,
$$
\widehat w(z;h) =\O(e^{(\Phi (z) +\varepsilon)/h}).
$$
Hence, by \eqref{uppbdPhi}, and taking $\varepsilon$ sufficiently small (and $\delta < 1/2$), at these points we obtain,
$$
\widehat w(z;h) =\O(e^{(\Phi_\delta (z) +\varepsilon - \delta_0^2 +\delta\delta_0^2)/h}) =\O(e^{\Phi_\delta (z) -\frac12 \delta_0^2)/h}).
$$
Summing up, we have proved the existence of a constant $\delta_2>0$ such that,
\be
\label{wbordOmega}
e^{-\Phi_\delta /h}\widehat w =\O(e^{-\delta_2/h})\,\, \mbox{ on } \partial\Omega.
\ee
Now, the idea is to perform estimates on $\widehat w$  in the space, 
$$
L^2_{\Phi_\delta}(\Omega):=L^2(\Omega\, ;\, e^{-2\Phi_\delta (z)/h}L(dz)),
$$
 where $L(dz)$ stands for the Lebesgue measure on $\C^n$. Before that, we go back to \eqref{systcarct2transf}, and, taking advantage of the ellipticity of 
$\widehat A_2$, we re-write it as,
\be
\label{systcarct2transfbis}
\left\{
\begin{aligned}
& \widehat w_1=\frac{1}{h}\widehat B_2 \,ihD_{z_n}\widehat w_2 \, ;\\
&\widehat A_2(-hD_{z_n}+i L)\widehat B_2\,  hD_{z_n}\widehat w_2 = h^2\widehat A_2\widehat A_1\widehat w_2,
\end{aligned}
\right.
\ee
where $\widehat B_2$ is a parametrix of $\widehat A_2$ near $(0,0)$, and where the identities hold in $H_{\Phi, 0}$.

We set,
\be
\widehat Q:=\widehat A_2(-hD_{z_n}+iL)\widehat B_2\, hD_{z_n},
\ee
and we plan to investigate the quantity,
\be
\label{defJ}
J:= \la \widehat Q\widehat w_2,\widehat w_2\ra_{L^2_{\Phi_\delta} (\Omega)} = \int_\Omega e^{-2\Phi_\delta (z)/h}Q\widehat w_2(z)\, \overline{\widehat w_2(z)} L(dz).
\ee
At first, we write,
$$
\widehat Q=(-hD_{z_n}+iL)\, hD_{z_n}+[\widehat A_2, -hD_{z_n}+iL]\widehat B_2\, hD_{z_n},
$$
and we approximate the pseudodifferential operators $[\widehat A_2, -hD_{z_n}+iL]\widehat B_2$ and $L$ by using Lemma \ref{approxpseudo} with $\phi =\Phi_\delta$. We obtain,
$$
\widehat Q\widehat w_2 =  \widehat Q_0\widehat w_2  + v,
$$
where $\widehat Q_0$  is the differential operator,
\be
 \widehat Q_0:= \left(-hD_{z_n} +i l_0 (z, \frac2{i}\nabla\Phi_{\delta})+ i(\nabla_\zeta l_0)(z, \frac2{i}\nabla\Phi_{\delta})(hD_z -\frac2{i}\nabla\Phi_{\delta})\right)hD_{z_n},
\ee
and with $v$ satisfying,
$$
\Vert e^{-\Phi_\delta /h}v\Vert_{L^2(\Omega)} \leq C_0h\Vert e^{-\Phi_\delta /h} hD_{z_n}\widehat w_2\Vert_{L^2(\Omega')}
$$
where $\Omega' $ is an arbitrarily small neighborhood of $\overline \Omega$, and where $C_0>0$ does not depend on the (sufficiently small) size of $\Omega$. In the same way as for \eqref{wbordOmega}, we see that,
$$
\Vert e^{-\Phi_\delta /h} hD_{z_n}\widehat w_2\Vert_{L^2(\Omega'\backslash \Omega)}=\O(e^{-\delta_2 /h}),
$$
and thus, we actually have,
\be
\label{estresteQ}
\Vert e^{-\Phi_\delta /h}v\Vert_{L^2(\Omega)} \leq C_0h\Vert e^{-\Phi_\delta /h} hD_{z_n}\widehat w_2\Vert_{L^2(\Omega)}+Ce^{-\delta_2 /h},
\ee
where  $C>0$ and $\delta_2$ may depend on the size of $\Omega$, but not $C_0$.
Therefore, we can write,
\be
\label{defJ}
J = \la \widehat Q_0\widehat w_2,\widehat w_2\ra_{L^2_{\Phi_\delta} (\Omega)} +J_1,
\ee
with,
\be
\label{estJ1.1}
|J_1| \leq \left(C_0h\Vert e^{-\Phi_\delta /h} hD_{z_n}\widehat w_2\Vert_{L^2(\Omega)}+Ce^{-\delta_2 /h}\right) \Vert e^{-\Phi_\delta /h} \widehat w_2\Vert_{L^2(\Omega)}.
\ee
We first study $J_0:= \la \widehat Q_0\widehat w_2,\widehat w_2\ra_{L^2_{\Phi_\delta} (\Omega)}$. Since $\overline{\widehat w_2(z)}$ is anti-holomorphic, we have $hD_z \overline{\widehat w_2(z)}=0$ and thus, an integration by parts leads us to,
\be
\label{defJ0}
J_0 = \int_{\Omega} {}^t \widehat Q_0 (e^{-2\Phi_\delta (z)/h}) \widehat w_2(z)\overline{\widehat w_2(z)}L(dz) + \O (e^{-\delta_2/h}),
\ee
where ${}^t \widehat Q_0$ is the formal adjoint of $\widehat Q_0$, and is given by,
$$
{}^t \widehat Q_0= -hD_{z_n}\left(hD_{z_n}+il_1(z)-i(hD_z+\frac2{i}\nabla\Phi_{\delta})\cdot l_2(z) \right),
$$
where we have set,
$$
l_1(z):= l_0 (z, \frac2{i}\nabla\Phi_{\delta})\quad ; \quad l_2(z):=(\nabla_\zeta l_0)(z, \frac2{i}\nabla\Phi_{\delta}).
$$
Using that $hD_z(e^{-2\Phi_\delta (z)/h}) = -\frac2{i}\nabla\Phi_{\delta}$,  and setting,
$$
l_3(z):= \sum_{j=1}^n i\frac{\partial}{\partial z_j}\left(\frac{\partial l_0}{\partial\zeta_j}(z, \frac2{i}\nabla\Phi_{\delta})\right),
$$
we find,
$$
\begin{aligned}
 {}^t \widehat Q_0 (e^{-2\Phi_\delta (z)/h})& =-hD_{z_n}\left( (-\frac2{i}\partial_{z_n}\Phi_{\delta}+il_1(z)+ihl_3(z))e^{-2\Phi_\delta (z)/h}\right) \\
 &=\left(4(\partial_{z_n}\Phi_{\delta})^2  +2l_1(z)\partial_{z_n}\Phi_{\delta}+hl_4(z)
 \right)e^{-2\Phi_\delta (z)/h},
 \end{aligned}
$$
with,
$$
l_4:=-2\partial_{z_n}^2\Phi_{\delta}  -\partial_{z_n}l_1 +2l_3\partial_{z_n}\Phi_{\delta}-h\partial_{z_n}l_3.
$$
Then, going back to \eqref{defPhidelta}, we compute,
\be
\label{dznPhi}
\begin{aligned}
& \partial_{z_n}\Phi_{\delta} =-\delta \Re z_n-\frac{i\mu}2 \Im z_n+\O(|z_n|^2);\\
& \partial_{z_n}^2\Phi_{\delta} =-\frac{\delta}2-\frac{\mu}4+\O(|z_n|);\\
& \nabla_{z'}\Phi_\delta =4\overline{z'},
 \end{aligned}
\ee
and thus, by Lemma \ref{approxl0},
$$
\begin{aligned}
& l_1(z) =(1-\delta)g(0)\Re z_n + \O\left(|z_n|(|z'|+|\rho|) + |z_n|^2\right);\\
& l_4(z) =\delta +\frac{\mu}2 -\frac{(1-\delta)g(0)}2 +\O(|z|+|\rho|+h).
\end{aligned}
$$
Therefore,
\be
\label{resymbole}
\begin{aligned}
& \Re\left( 4(\partial_{z_n}\Phi_{\delta})^2  +2l_1(z)\partial_{z_n}\Phi_{\delta}+hl_4(z)\right)\\
& \hskip 1cm =-\mu^2|\Im z_n|^2-2\delta (1-3\delta)g(0)|\Re z_n|^2-\frac{(1-\delta)g(0)-\mu-2\delta}2 h\\
& \hskip 3cm +\O\left (|z_n|^3+|z_n|^2(|z'|+|\rho|)+h(|z|+|\rho|+h)\right).
\end{aligned}
\ee
Now, we see on \eqref{resymbole} that if we choose $\mu <g(0)$, $\delta$ small enough, and the diameter of $\Omega$ small enough, then there exists a constant $C_1>0$ (independent of the diameter of $\Omega$) such that, on $\Omega$,
\be
\label{negativesymbol}
\Re\left( 4(\partial_{z_n}\Phi_{\delta})^2  +2l_1(z)\partial_{z_n}\Phi_{\delta}+hl_4(z)\right)\leq -\frac1{C_1}(|z_n|^2 + h).
\ee
Of course, this estimate is the analogue of \eqref{toy}, obtained for the toy-model.

Going back to \eqref{defJ}, we deduce frorn \eqref{negativesymbol},
\be
\label{estJ0}
\Re J_0 \leq -\frac{h}{C_1}\Vert \widehat w_2\Vert_{L^2_{\Phi_\delta} (\Omega)}^2 + Ce^{-\delta_2/h}.
\ee

Now, concerning $J_1$, we write,
$$
\begin{aligned}
\Vert  hD_{z_n}\widehat w_2\Vert_{L^2_{\Phi_\delta} (\Omega)}^2 & =\int_\Omega  -hD_{z_n}(e^{-2\Phi_\delta /h})\, \widehat w_2\, \overline{hD_{z_n}\widehat w_2} \, L(dz) + \O(e^{-\delta_2 /h})\\
&\leq 2\Vert  (\partial_{z_n}\Phi_\delta)\widehat w_2\Vert_{L^2_{\Phi_\delta}(\Omega)}\Vert  hD_{z_n}\widehat w_2\Vert_{L^2_{\Phi_\delta} (\Omega)} + Ce^{-\delta_2 /h},
\end{aligned}
$$
and therefore, by \eqref{dznPhi},
$$
\Vert  hD_{z_n}\widehat w_2\Vert_{L^2_{\Phi_\delta} (\Omega)}\leq C_2\Vert  z_n\widehat w_2\Vert_{L^2_{\Phi_\delta} (\Omega)}+ C'e^{-\delta_2' /h},
$$
with $C_2$, $C'$, $\delta_2' $ positive, and $C_2$ independent of the (sufficiently small) size of $\Omega$.

Inserting into \eqref{estJ1.1}, and recalling \eqref{defOmega}, we deduce,
\be
\label{estJ1.2}
|J_1| \leq C_3\delta_0 h\Vert \widehat w_2\Vert_{L^2_{\Phi_\delta} (\Omega)}^2 + C''e^{-\delta_2''/h}
\ee
with $C_3$, $C''$, $\delta_2'' $ positive, and $C_3$ independent of $\delta_0$.
Putting together \eqref{estJ0} and \eqref{estJ1.2}, and  choosing $\delta_0$ sufficiently small, we finally obtain the existence of a (new) constant $C>0$ such that 
\be
\la \widehat Q\widehat w_2,\widehat w_2\ra_{L^2_{\Phi_\delta} (\Omega)}\leq -\frac{h}{C}\Vert \widehat w_2\Vert_{L^2_{\Phi_\delta} (\Omega)}^2 + Ce^{-\delta_2/h}.
\ee
Finally, going back to \eqref{systcarct2transfbis}, we deduce,
$$
h\Vert \widehat w_2\Vert_{L^2_{\Phi_\delta} (\Omega)}^2=\O(h^2)\Vert \widehat w_2\Vert_{L^2_{\Phi_\delta} (\Omega)}^2+\O(e^{-\delta' /h})
$$
with $\delta'>0$ constant, and thus, for $h$ sufficiently small,
$$
\Vert \widehat w_2\Vert_{L^2_{\Phi_\delta} (\Omega)}^2=\O(e^{-\delta' /h}).
$$
Since $\Phi_\delta (0) =|\rho (h)|^2 \to 0$ as $h\to 0$, by taking $\Omega'\subset \Omega$ a sufficiently small complex neighborhood of 0, we obtain
$$
\Vert \widehat w_2\Vert_{L^2 (\Omega')}=\O(e^{-\delta' /4h}),
$$
and by \eqref{caractMS} (and the fact that $\Phi (0)\leq|\rho|^2$), this implies that $(0,0)\notin MS(w_2)$. Using the first equation of \eqref{systcarct2transfbis}, we also deduce that $(0,0)\notin MS(w_1)$, and thus, finally, that $(0,0)\notin MS(w)$. Hence, Proposition \ref{caracdouble} is proved.

\section{Propagation up to the well $U$}

Thanks to Proposition \ref{caracdouble}, we know that $e^{\varphi_2/h}u$ is exponentially small near $x^{(1)}$, and since $\varphi_2(x^{(1)})=\varphi_1(x^{(1)})=\varphi (x^{(1)})$, we conclude that the same is true for $e^{\varphi/h}u$. Then, proceeding as for Proposition \ref{estgamma1}, this information can be propagated along $\gamma^{(2)}$, up to a small neighborhood of $x^{(2)}$. But in $x^{(2)}$, the situation is completely similar to that in $x^{(1)}$, with the only difference that  the roles of the two components $u_1$ and $u_2$ are exchanges. But thanks to Assumption 6, 
the arguments of the previous section can be first repeated identically, and we conclude that $e^{\varphi/h}u$ is exponentially small near $x^{(2)}$. Iterating $N$ times this procedure, we arrive to the existence of $\delta >0$ such that,
\be
e^{\varphi/h}u =\O(e^{-\delta /h})\, \mbox{ uniformly near } x^{(N)}.
\ee
Again, we can propagate this information along $\gamma^{(N+1)}$ up to a small neighborhood of $y_0\in \partial U$, and it remains to show that, in that case, $u$ is exponentially small near $y_0$ (since $\gamma \in G_1$, this will be in contradiction with Assumption 4, and the proof of Theorem \ref{mainth} will be completed).

We adopt the strategy used in \cite[Theorem 1.1]{Ma2} (see also \cite[Section 7]{DaMa}), with some modifications due to the fact that we deal with a matrix operator, and with some technical improvements (where it has been possible) and corrections (where it has been necessary).

At first, proceeding as in \cite[Section 3]{Ma2}, we slightly modify the  definition of all the geometric quantities by changing the potentials $V_1$ and $V_2$ into $V_1-\Re\rho$ and $V_2-\Re\rho$. Since $\Re\rho \to 0$ as $h\to 0$, this does not modify the geometric situation (see \cite[End of Section 1]{Ma2}), and we still denote by $U$, $d$, $\gamma$, $y_0$ the corresponding quantities after this change. From now on, the quantity $E:=\Re \rho$ is considered as an extra small parameter, with respect to which all the estimates will be uniform.

We also fix $y_1, z_1\in \gamma^{(N+1)}$ such that $d(y_1,y_0)<d(z_1,y_0)$ and, for $(y,z)$ sufficiently close to $(y_1, z_1)$, we consider the function,
$$
F(y,z):= d(y,z)-d(z,U).
$$
Then $F$ is analytic, and it is solution to,
$$
(\nabla_y F(y,z))^2 = V_1(y)-E.
$$
Moreover, $F$ can be analytically continued with respect to $y$ along $\gamma$ in the direction of $y_0$. At $y_0$, the map $y\mapsto F(z,y)$ develops a singularity, typically in ${\rm dist} (y,{\mathcal C}_z)^{3/2}$, where ``${\rm dist}$'' stands for the Euclidean distance, and ${\mathcal C}_z$ (the caustic set) is an analytic hypersurface tangent to $\partial U$ at some point $y_0(z)$. However, using a technique form \cite[Section 10]{HeSj2}, it is possible to go round this singularity in the complex domain and to extend analytically $F(y,z)$ up to $\{ V_1(y) < E \}$, avoiding ${\mathcal C}_z$. Depending on the way we turn around ${\mathcal C}_z$, we obtain two possible (complex-valued) extensions $F_\pm (y,z)$, that satisfy,
\be
\label{eikonFpm}
\begin{aligned}
& (\nabla_y F_\pm (y,z))^2 = V_1(y)-E ;\\
& F_-(y,z) = \overline{F_+}(y,z)\quad ;\quad \Im F_+(y,z) \sim {\rm dist}(y, {\mathcal C}_z))^{\frac32}.
\end{aligned}
\ee

Furthermore, ${\mathcal C}_z$ has a contact of order exactly 2 with $\partial U$, at the unique point $y_0(y)$ that is connected to $z$  through a minimal geodesic $\gamma_z \subset \R^n\backslash U$ (see \cite[Lemmas 3.1 and 3.2]{Ma2}). Thanks to this, if we fix some small real-analytic hypersurface $\Gamma \subset \overset{\circ}U$ transversal to the $x$-projection  of the bicharacteristic $\bigcup_{t\in \R} \exp tH_{p_1} (y_0,0)$, together with 
a real-analytic hypersurface $\Sigma \subset \R^n\backslash U$ transversal to $\gamma$ at $z_1$, then, one can prove the existence of a constant $C>0$ such that, for all $(y,z)\in \Gamma\times \Sigma$,
\be
\label{HessReF}
\Re F_\pm (y,z) \geq \frac1{C}|y-w(z)|^2,
\ee
where $w(z)$ is the point of intersection between $\Gamma$ and the $x$-projection $\alpha_z$ of the bicharacteristic $\bigcup_{t\in \R} \exp tH_{p_1} (y_0(z),0)$ (see \cite[Formula (3.5)]{Ma2}).

In addition, if $(y',z')$ stands for local coordinates on $\Gamma\times\Sigma$, we also have (see \cite[Section 3]{Ma2}),
\be
\nabla_{y'}\Re F_\pm(w(z), z)=0\quad ;\quad \det \nabla_{y'}\nabla_{z'} \Re F_\pm (z_1, w(z_1)) \not=0.
\ee

Since $P_2$ is elliptic near $\partial U \times \{0\}$, and $F_\pm$ satisfies \eqref{eikonFpm}, it can be used to construct asymptotic solutions $v_\pm =v_\pm (y,z;h)$ to the equation $(P(y,hD_y)-E)v_\pm \sim 0$. More precisely, for $y$ away from ${\mathcal C}_z$, $v_\pm$ is constructed of the form $a_\pm (y,z;h)e^{-F_\pm (y,z)/h}$, with $a_\pm$ a (vector-valued) classical analytic symbol, while, for $y$ close to ${\mathcal C}_z$, $v_\pm$ can be represented as a Airy-type integral (see \cite[Section 3]{Ma2}), and we have,
\be
\label{eqasympt}
e^{\Re F_\pm (y,z)/h }(P(y,hD_y)-E)v_\pm (y,z) =\O(e^{-\delta /h},
\ee
where $\delta$ is some positive constant, and where $z$ may vary in a complex neighborhood of $z_1$, and $y$ in a (real) neighborhood $\Omega'$ of $y_0$, that may be taken tubular around $\gamma \cup \alpha_{z_1}$, with boundary $\Gamma \cup \Sigma' \cup {\mathcal T}$, where ${\mathcal T}\cap (\gamma \cup \alpha_{z_1})=\emptyset$ and $\Sigma' \subset \R^n\backslash U$.

In addition, since $P_2$ is elliptic there, the symbol appearing in \eqref{eqasympt} is of the form,
$$
a_\pm =\left( \begin{array}{c}
a_1^\pm\\
ha_2^\pm
\end{array}\right),
$$
where $a_1^\pm$ and $a_2^\pm$ are 0-th order classical analytic symbols, with $a_1^\pm$ elliptic.

The next step consists in applying the Green formula to the (exponentially small) quantity $\la (P-E)u, v_\pm (y,\cdot) \ra_{L^2(\Omega')}-\la u, (P-E)v_\pm (y,\cdot) \ra_{L^2(\Omega')}$. Using the same notations as in \eqref{interW}-\eqref{formHS}, and setting $v_\pm =(v_1^\pm, v_2^\pm)$, this gives,
$$
\int_{\partial\Omega'}\left( \frac{\partial u}{\partial\nu}\overline{v_\pm} - u\frac{\partial \overline{v_\pm}}{\partial\nu}+(r_1\cdot \nu)(u_1\overline{v_2^\pm}-u_2\overline{v_1^\pm} )  \right)ds=\O(e^{-\delta'/h}),
$$
with $\delta'>0$ constant.

Then, using the fact that $e^{\varphi /h}u$ is exponentially small on $\Sigma'$, together with  the properties of $F_\pm$ (in particular the fact that $\Re F_\pm (y,z_1) + d(y,U)$ remains non negative, and is positive for $z$ away from  $\gamma\cup \alpha_{z_1}$: see \cite[Formula (4.2)]{Ma2}), we obtain,
\be
\label{intGammaexp0}
\int_{\Gamma}\left( \frac{\partial u}{\partial\nu}\overline{v_\pm} - u\frac{\partial \overline{v_\pm}}{\partial\nu}+ (r_1\cdot \nu)(u_1\overline{v_2^\pm}-u_2\overline{v_1^\pm} )  \right)ds=\O(e^{-\delta'/h}).
\ee
Then, using the ellipticity of $a_1^\pm$, we can write,
\be
\label{v2vsv1}
\begin{aligned}
& v_2^\pm = h\frac{a_2^\pm}{a_1^\pm} v_1^\pm =: hf_\pm v_1^\pm;\\
& h\frac{\partial {v_1^\pm}}{\partial\nu} = -\left(a_1^\pm \frac{\partial {F_\pm}}{\partial\nu}+h\frac{\partial {a_1^\pm}}{\partial\nu} \right)e^{-F_\pm /h} =: g_1^\pm v_1^\pm;\\
& h\frac{\partial {v_2^\pm}}{\partial\nu} = hf_\pm g_1^\pm v_1^\pm + h^2\frac{\partial {f}}{\partial\nu}v_1^\pm := hg_2^\pm v_1^\pm,
\end{aligned}
\ee
where $f_\pm$, $g_j^\pm$ are 0-th order classical symbols, and $g_1^\pm = -\frac{\partial {F_\pm}}{\partial\nu}+\O(h)$.
Inserting \eqref{v2vsv1} into \eqref{intGammaexp0} and multiplying by $h$, we find,
\be
\label{intGammaexp1}
\int_{\Gamma}  (G_\pm u)\overline{v_1^\pm} ds=\O(e^{-\delta'/h}),
\ee
with,
\be
\label{Gpmu}
G_\pm u:=h\frac{\partial u_1}{\partial\nu}+h^2 \overline{f_\pm}\frac{\partial u_2}{\partial\nu}-\overline{g_1^\pm}u_1 -h\overline{g_2^\pm}u_2+h(r_1\cdot \nu)(h\overline{f^\pm}u_1-u_2).
\ee

Now, we specify a little bit better our choice of $\Gamma$. We set, 
$$
f_2(y):= \Im F_+(y, z_1).
$$
Then, it can be seen that, inside $U$,  the curve $\alpha_{z_1}$ actually coincides with $\widetilde \gamma:= \bigcup_{t>0} \exp t\nabla f_2 (y_0)$ (indeed, using \eqref{eikonFpm}, we see that $\nabla_y \Re F_\pm (y,z_1)=0$ along $\widetilde \gamma$, and, setting $x(t):= \exp 2t \nabla f_2 (y_0)$ and $\xi (t):= \frac12 \dot x (t)$, we deduce that $(x(t), \xi (t))$ is solution to the Hamilton system associated with $p_1$).

Now, we make depend the hypersurface  $\Gamma$ of a parameter $t$ close enough to some fixed $t_0>0$, and we choose it in such a way that $\Gamma =\Gamma_t$ is orthogonal to $\widetilde \gamma$ at $y_t:= \exp t\nabla f_2 (y_0)$.

Observing that the constant $\delta'$ appearing in \eqref{intGammaexp1} can be taken independent of $t$, we multiply this expression by $e^{-(s -t)^2/h}$ (where $s$ is some extra complex parameter), and integrate with respect to $t$ around $t_0$. We obtain,
\be
\label{intomegaexp0}
\int_{\omega}(G_\pm u)\theta^\pm dy'dt=\O(e^{(\Im s)^2-\delta')/h}),
\ee
where  $\theta_\pm := J\overline{v_1^\pm}e^{-(s -t)^2/2h}$,  $J$ being the Jacobian coming from the change of coordinates $y\mapsto (y',t)$
(with $y'=$  Euclidean coordinates orthogonal to $\widetilde\gamma$ at $y_{t_0}$), and $\omega$ is a small volume around $y_{t_0}$.

By construction, $\theta_\pm$ is of the form,
$$
\theta_\pm (y',t; z', \tau) = b^\pm e^{i\psi_\pm (y',t; z',s)/h},
$$
where $b$ is an elliptic analytic classical symbol,  $\psi_\pm := iF_\pm +\frac{i}2(s -t)^2$, and $z'$ stands for local coordinates on the complexified  $\Sigma_\C$ of $\Sigma$.

In particular, using \eqref{HessReF} and the fact that $\nabla_{y} \Re F_\pm (y_{t_0}, z_1) =0$ and (thanks to our choice of $\Gamma_t$) $\nabla_{y'}f_2(y_{t_0})=0$, we see that, for any $\tau_0\in\R$, $\psi_\pm$ satisfies,
\be
\label{propopsipm}
\begin{aligned}
& \nabla_{(y',t)}\psi_\pm (0, t_0; z'_1, t_0 -i\tau_0) = -(0,\tau_0);\\
& \Im \nabla_{(y',t)}^2 \psi_\pm (0, t_0; z'_1, t_0 -i\tau_0) >0;\\
& \det \nabla_{(y',t)}\nabla_{z',\tau} \psi_\pm (0, t_0; z'_1, t_0 -i\tau_0)\not=0;\\
& \Phi_\pm (z',\tau):= \sup_{(y',t) \mbox{ real}} \left( -\Im \psi_\pm (0,t_0; z_1',t_0 -i\tau_0)\right)= \tau_0^2.
\end{aligned}
\ee
Therefore, according to the general theory of \cite{Sj} (see also \cite[Appendix a]{HeSj3}, the estimate \eqref{intomegaexp0} expresses the fact that, for any $\tau_0\in \R$, we have,
\be
\label{pasMS1}
(y_{t_0}; 0,\tau_0) \notin MS(G_\pm u). 
\ee
Now, since $P_2$ is elliptic on $\partial U\times \{0\}$, we can construct a microlocal parametrix of $P_2-\rho$ near $(y_{t_0},0)$, and from the system $(P-\rho)u=0$ we deduce,
\be
\label{u2Qu1}
u_2 =hQu_1
\ee
where $Q$ a 0-th order analytic pseudodifferential operator. Therefore, we see on \eqref{Gpmu} that $G_\pm u$ can be microlocally re-written as,
$$
G_\pm u \sim (1+hA_\pm)h\frac{\partial u_1}{\partial \nu}-(\overline{g_1^\pm}+hB_\pm)u_1,
$$
where $A_\pm$ and $B_\pm$ are 0-th order analytic pseudidifferential operators (here, the symbol ``$\sim$'' means that the identity is valid in a microlocal sense only, and we refer to \cite[Appendix]{HeSj3} for details on this notion).

Since elliptic operators do not change the microsupport, by taking a parametrix of $(1+hA_\pm)$ we see that \eqref{pasMS1} is equivalent to,
\be
\label{pasMS2}
(y_{t_0}; 0,\tau_0) \notin MS\left( h\frac{\partial u_1}{\partial \nu}-(\overline{g_1^\pm}+h\widetilde B_\pm)u_1\right),
\ee
still with $\widetilde B_\pm$ a 0-th order analytic pseudidifferential operator. Taking the difference between the two functions in \eqref{pasMS2} (that is, the one with the $+$ index and the one with the $-$ index), we deduce,
\be
\label{pasMS3}
(y_{t_0}; 0,\tau_0) \notin MS\left( (\overline{g_1^+ -g_1^-}+h\widetilde B_+ -h\widetilde B_-))u_1\right),
\ee
Now, we observe on \eqref{v2vsv1} that the symbols $g_1^+$ and $g_1^-$ are such that $\Im g_1^\pm$ is  elliptic (this is because $\partial_\nu f_2(y_{t_0})\not=0$), and  $\Im g_1^- = -\Im g_1^+ + \O(h)$. In particular, $g_1^+ -g_1^-$ is elliptic, and \eqref{pasMS3} implies that, for all $\tau_0\in \R$, we have,
\be
\label{pasMS4}
(y_{t_0}; 0,\tau_0) \notin MS(u_1).
\ee
Since, for all $\varepsilon >0$,  $u_1$ is $\O(e^{\varepsilon /h})$ uniformly near $y_0$, and $u_1$ is solution to $(P_1-\rho-h^2W^*)u_1=0$ there, by standard results on the propagation of the micro-support, we know that $MS(u_1)$ is invariant under the Hamilton flow of $p_1$. In particular, taking $\tau_0=\sqrt{ E-V_1(y_{t_0})}$ (so that $(y_{t_0}; 0,\tau_0)$ is in $\{ p_1 =E\}$ and it can be joined to $(y_0,0)$ by a bicharacteristic of $p_1$), we conclude that $(y_0,0)\notin MS(u_1)$.
By \eqref{u2Qu1} this also implies that $(y_0,0) \notin MS(u_2)$, and therefore, 
\be
\label{pasMS5}
(y_0,0) \notin MS(u).
\ee
Recall that here, $y_0$ is the point of $\partial U=\partial U(E):=\{ V_1 =E\}$ where the modified geodesic $\gamma =\gamma (E)$ reaches $U(E)$, and that, in \eqref{pasMS5}, $E$ is considered as an extra ($h$-independent) parameter. However, since all the previous estimates are uniform with respect to $E$ small enough, and since $MS(u)$ is a closed subset of $\R^{2n}$, we can particularize to $E=E(h):=\Re \rho (h)$, and finally obtain (this time with the original quantities),
$$
(\gamma \cap \partial U)\times \{0\} \,\,  \cap \, \, MS(u) \, =\, \emptyset.
$$
But, by Remark \ref{remhypMS} (and the fact that $u$ and $u_0$ are exponentially close to each other near $y_0$), this is in contradiction with Assumption 4, and therefore Theorem \ref{mainth} is proved.

\section{Proof of Theorem \ref{convmainth}}

The fact that Assumption 4 is not satisfied means that, for any $\gamma \in G$, there exists $\delta_0 >0$ and a neighbourhood ${\mathcal W}$ of $\gamma \cap \partial U$, such that 
\be
\label{exponsmall}
\Vert u\Vert_{L^2({\mathcal W})\oplus L^2({\mathcal W})} =\O(e^{-\delta_0 /h}).
\ee
In addition, since $\partial U$ is compact, the value of $\varepsilon_0$ can be taken independent of $\gamma$. 

By \eqref{formimrho}, it is enough to prove that $v:=e^{\varphi /h}u$ remains exponentially small first along $\gamma$, and then along $\Pi_x \exp tH_{p_2}((\gamma \cap \partial {\mathcal M})\times\{0\})$ ($t$ small enough).

At first, using \eqref{exponsmall}, we see that for $t>0$ small enough  (that is, for $\gamma (t)$ is sufficiently close to $\gamma\cap U$), we have $(\gamma(t),0)\notin MS(v)$. Then, by the same argument of propagation used for \eqref{propag}, we see that this property remains valid for all $t\in [0, t_1)$, where $t_1$ is the first time for which $\gamma (t)$ reaches the crest $\partial\Omega_0$. Hence, using $MS(v)\cap [(\gamma\cap \Omega_0)\times \R^n] \subset \{ \xi =0\}$, we deduce the existence, for any $t'_1<t_1$, of $\delta_1=\delta_1(t'_1)>0$ and a neighbourhood ${\mathcal W}_1$ of $\gamma(t'_1)$ such that,
\be
\label{exponsmall1}
\Vert e^{\varphi /h}u\Vert_{L^2({\mathcal W}_1)\oplus L^2({\mathcal W}_1)} =\O(e^{-\delta_1 /h}).
\ee
Fixing $t'_1<t_1$ sufficiently close to $t_1$, and taking local Euclidean coordinates $(x',x_n)$ centered at $\gamma (t_1)$ such that $\dot \gamma (t_1)$ is in $\{ x'=0,\, x_n>0\}$,  we consider the weight-function,
\be
\label{defpsiAgmon}
\psi (x) := \varepsilon_1 (1- \varepsilon_2^{-2}|x'|^2)\varphi_0+ (1-\varepsilon_3)\varphi (x),
\ee
where $\varphi_0:= \varphi (\gamma (t_1))$ and $\varepsilon_1, \varepsilon_2, \varepsilon_3 >0$ are small enough. Denoting by $(x'_1, x_n^1)$ the corresponding coordinates of $\gamma (t_1')$, we have $x_n^1 <0$, $x_n^1\sim t'_1-t_1$, $x_1'=\O( (t_1-t'_1)^2)$, and  $\varphi (0,x_1') <\varphi (0)$.

Then, , we consider the open set,
$$
\omega := \{ |x'|<\varepsilon_2\, ;\, |x_n| < |x_n^1| \},
$$
and, setting $\varphi_+:= \min \{ \varphi (x)\, ; \, x_n=-x_n^1\, , \, x\in \omega\} >\varphi_0$, we observe,
\begin{itemize}
\item On $|x'| =\varepsilon_2$: We have $\psi (x) = (1-\varepsilon_3)\varphi (x) < \varphi (x)$;
\item On $x_n = x_n^1$: $\psi (x) \leq \varepsilon_1\varphi_0 +(1-\varepsilon_3)\varphi (x)\leq \varphi(x) + \delta_1$ if we have chosen $\varepsilon_1\leq \delta_1/\varphi_0$;
\item On $x_n =-x_n^1$: $\psi (x) \leq \varepsilon_1\varphi_0 +(1-\varepsilon_3)\varphi (x)< \varphi (x)$ if we have chosen $\varepsilon_1< \varepsilon_3(\varphi_+/\varphi_0)$;
\item On $\omega_\varepsilon$: $|\nabla \psi (x)|^2 = (1-\varepsilon_3)^2|\nabla \varphi (x)|^2 + \O(|\varepsilon_1/\varepsilon_2|^2)\leq (1-\varepsilon_3)|\nabla\varphi(x)|^2$ if we have chosen $\varepsilon_3 <<1$ and $\varepsilon_1 << \varepsilon_2$;
\item $\psi (0) = (1+\varepsilon_1 -\varepsilon_3)\varphi_0 > \varphi_0$ if we have chosen $\varepsilon_1 > \varepsilon_3$.
\end{itemize}
In particular, taking $\varepsilon_1 \in (\varepsilon_3,  \varepsilon_3(\varphi_+/\varphi_0))$ and $\varepsilon_3 << \varepsilon_2 << 1$,  on $\partial \omega$ we have $e^{\psi /h}u =\O(1)$. Thus, performing Agmon estimates on $\omega$ (by using Green's formula), we easily conclude,
$$
\Vert h\nabla (e^{\psi /h} u)\Vert_{L^2(\omega)} +\Vert e^{\psi /h} u\Vert_{L^2(\omega)} =\O(1).
$$
Since $\psi (0) >\varphi (\gamma (t_1))$, we deduce the existence of $\delta_2 >0$ such that $e^{\varphi /h}u =\O(e^{-\delta_2/h}$ near $\gamma (t_1)$.

Again, as before we can propagate this information along $\gamma$ up to the next point where $\gamma$ crosses $\partial\Omega_0$, where the same argument applies. Iterating the procedure, we reach in this way any point of gamma, except its final end-point $\{ x_0\} :=\gamma\cap\partial{\mathcal M}$.

In particular, if we fix $\widetilde x_0 \in \gamma \backslash \{ x_0\}$ close enough to $x_0$, then there exists $\delta>0$ such that $e^{(\varphi + 2\delta)/h} u$ is bounded near $\widetilde x_0$.

In a spirit similar to that of \eqref{defpsiAgmon}, for $\varepsilon_1, \varepsilon_2, \varepsilon_3 >0$ small enough, we set,
$$
\widetilde\varphi (x) = \min \left(\varphi (x) + \varepsilon_1\, ; \, S_0 + d(x,{\mathcal M}) \right),
$$
and we define, 
$$
\widetilde \psi (x) := \varepsilon_2 (1- \varepsilon_3^{-2}|x'|^2)S_0+ (1-\varepsilon_2)\widetilde \varphi (x),
$$
where, this time, $(x',x_n)$ are Euclidean coordinates centered at $x_0$, with $\{ x_n=0\} = T_{x_0}\partial {\mathcal M}$ and $\widetilde x_0 \in \{ x_n<0\}$.

We have,
\begin{itemize}
\item On $\{ |x'| = 2\varepsilon_3\}$: $\widetilde \psi (x) =-3\varepsilon_2S_0 +(1-\varepsilon_2)\widetilde \varphi (x)\leq \varphi (x) + \varepsilon_1-3\varepsilon_2S_0 $, and thus, $\widetilde \psi (x)\leq \varphi (x)$ if $\varepsilon_1 \leq3\varepsilon_2S_0$;
\item On $\{ x_n = \widetilde x_n^0\}$ (where $\widetilde x_n^0<0$ is the last coordinate of $\widetilde x_0$): $\widetilde \psi (x)\leq \varepsilon_2S_0 +\varphi (x)+\varepsilon_1\leq \varphi(x)+\delta$ if we have taken $\varepsilon_2S_0+\varepsilon_1 \leq \delta$;
\item On $\{ x_n =| \widetilde x_n^0|\}$: $\widetilde \psi (x)\leq S_0$;
\item On $\{ |x'|< 2\varepsilon_3\, , \, |x_n| < | \widetilde x_n^0|\}$: $|\nabla \widetilde\psi(x)|^2 = (1-\varepsilon_2)^2|\nabla\widetilde\varphi (x)|^2 +\O(|\varepsilon_2/\varepsilon_3|^2)\leq (1-\varepsilon_2)|\nabla\widetilde\varphi (x)|^2$ if we have taken $\varepsilon_2 << \varepsilon_3$;
\item On $\{ \varphi (x) + \varepsilon_1 = S_0 + d(x,{\mathcal M}\}\cap \{ |x'|\leq \varepsilon_3\}$: Since $d(x,{\mathcal M}\geq S_0 -\varphi (x)$, on this set we have $\varphi (x) \geq S_0 -\varepsilon_1/2$, and thus $\widetilde \varphi (x) \geq S_0+\varepsilon_1/2$. Therefore $\widetilde \psi (x)\geq (1-\varepsilon_2)(S_0+\varepsilon_1/2)$, and then, if we have taken $\varepsilon_1 >2\varepsilon_2S_0/(1-\varepsilon_2)$, we obtain $\widetilde \psi (x)>S_0$ there.
\end{itemize}

Hence, using the fact that $|\nabla\widetilde\varphi (x)|^2\leq V_2(x)$ almost everywhere, and performing as before Agmon estimates on $\widetilde \omega:= \{ |x'|< 2\varepsilon_3\, , \, |x_n| < | \widetilde x_n^0|\}$, we obtain,
$$
\Vert h\nabla (e^{\widetilde\psi /h} u)\Vert_{L^2(\widetilde\omega)} +\Vert e^{\widetilde\psi /h} u\Vert_{L^2(\widetilde\omega)} =\O(e^{\varepsilon /h}).
$$
for all $\varepsilon >0$, and thus, thanks to the last property of $\widetilde\psi$, denoting by $\omega_1$ a small enough neighborhood of $\{ \varphi (x) + \varepsilon_1 = S_0 + d(x,{\mathcal M}\}\cap \{ |x'|\leq \varepsilon_3\}\cap \widetilde\omega$,
$$
\Vert u\Vert_{H^1(\omega_1)} =\O(e^{-(S_0+\delta_1)/h}),
$$
for some constant $\delta_1>0$.

This can be done along any $\gamma \in G$, and since we already know that $e^{S_0/h}u$ is exponentially small near $\partial{\mathcal M}\backslash \bigcup_{\gamma\in G} (\gamma \cap \partial{\mathcal M})$, we conclude that there exists a neighborhood ${\mathcal V}$ of $\partial{\mathcal M}$ such that, for any $\mu >0$ small enough, $e^{S_0/h}u$ (together with all its derivatives) is exponentially small on ${\mathcal V}\cap \{ V_2=-\mu\}$.

Then, setting $\Omega_2:=\widehat I \cap \{ V_2<-\mu\}$ and applying the Stokes formula on $\Omega_2$, we obtain,
\be
\label{formHSbis}
(\Im \rho)\Vert u\Vert^2_{L^2(\Omega_2)} = -h^2\Im  \int_{\partial\Omega_2} \frac{\partial u}{\partial\nu}\cdot \overline{u}ds + h^2\Im\int_{\partial\Omega_2}(r_1\cdot \nu)u_2\overline{u_1}ds,
\ee
and the result follows.

\section{About examples}
In contrast with the case of highly excited shape resonances (see \cite{DaMa}) where examples are not easy to construct, here, because of the presence of two potentials, the task is simpler. 

A first type of examples can be made by considering a rotationally invariant potential $V_1$. In that case, it is well known that the eigenvalues of $P_1$ corresponding to rotationally invariant eigenfunctions are separated by a gap of order $h^2$ from the rest of the spectrum. Since the eigenvalues of $\widetilde P$ (defined in \eqref{Ptilde}) coincide with those of $P_1$ up to some $\O(\Vert W\Vert^2 h^2)$, if the interaction $W$ is sufficiently small, the same gap will hold for the corresponding eigenvalues of $\widetilde P$, so that  Assumption 3 will be satisfied for them. In addition, one can see that the corresponding eigenfunctions of $\widetilde P$ are close, up to $\O(h^2)$, to those of $P_1$. On the other hand, the normalized rotationally invariant eigenfunctions of $P_1$ can be constructed by means of one dimensional (WKB or so) methods (see, e.g., \cite{Ya}), and near the boundary of the well $U$ (that corresponds to turning points), it can be seen that they are of size at least 1. Therefore, the same will be true for the corresponding eigenfunctions of $\widetilde P$, and Assumption 4 will be satisfied, too. Concerning Assumption 5, it strongly depends on the choice of the potential $V_2$, and one can for instance take it in such a way that the crest of the cirque realizes its (Euclidean) distance with the well at only one point, with sufficient non degeneracy in order to satisfy Assumption 5.

Another kind of examples may be given by taking $V_1$ with separated variables. For instance, in dimension 2, one may take $V_1(x,y) = v_1(x) + v_2(y)$, where $v_2^{-1}(0)$ is a bounded interval with non empty interior, while $v_1^{-1}(0) =\{0\}$. In that case, as  in \cite[Section 7]{Ma2}, one can construct sequences of values of $h$ for which some eigenvalues of $P_1$ that are close to 0 are separated by a gap of order $h^3$ from the rest of its spectrum, and the corresponding eigenfunctions are concentrated on $x=0$. If in addition the interaction $W$ vanishes at $x=0$, then, the arguments of the previous paragraph can be repeated, and give other examples where Theorem \ref{mainth} applies, for instance if $V_2$ is chosen in such a way that the minimal geodesics between the well and the sea starts on $x=0$. If, on the contrary, $V_2$ is chosen in such a way that any minimal geodesic $\gamma$ between the well and the sea starts away from $x=0$, then the eigenfunction is exponentially small on $\gamma\cap U$, and Assumption 4 is not satisfied anymore. In that case, Theorem \ref{convmainth} applies and shows that the width of the corresponding resonance is exponentially smaller that $e^{-2S_0/h}$.

\appendix
\vskip 1cm

\bigskip


{}

\end{document}